\documentclass[12pt]{amsart}
\usepackage[colorlinks=true,citecolor=blue,linkcolor=blue,urlcolor=black]{hyperref}
\usepackage{latexsym,amsmath,amssymb}
\usepackage{amscd}
\usepackage{multimedia}
\usepackage{mathtools} 
\usepackage{xparse} 
\usepackage{accents}
\usepackage{color}
\usepackage{a4wide}
\usepackage{soul} 
\usepackage{multicol}
\usepackage[colorinlistoftodos,prependcaption,textsize=tiny]{todonotes}

\title[Boundary regularity]{Boundary regularity for minimizing biharmonic maps}
\author{Katarzyna Mazowiecka}

\address[Katarzyna Mazowiecka]{
Institute of Mathematics,
University of Warsaw,
Banacha 2,
02-097 Warszawa, Poland
\newline
\&
Universit\'e catholique de Louvain, 
Institut de Recherche en Math\'ematique et Physique, 
Chemin du Cyclotron 2 bte L7.01.01, 1348 Louvain-la-Neuve, Belgium}
\email{katarzyna.mazowiecka@uclouvain.be}

plus 6pt minus 12pt
\newcommand{\R}{\mathbb{R}}  
    
\newcommand{\Z}{\mathbb{Z}}
\newcommand{\N}{\mathbb{N}} 
\renewcommand{\S}{\mathbb{S}}

\newcommand{\n}{\mathcal{N}}

\DeclareMathOperator{\dist}{dist}

\def\w{\widetilde}
\def\wu{\widetilde u}
\newcommand{\ve}{\varepsilon}
\newcommand{\eps}{\varepsilon}
\newcommand{\vp}{\varphi}
\newcommand{\h}{\mathcal{H}}

\newcommand{\brac}[1]{\left({#1}\right)}

\newcommand{\norm}[1]{\left\|{#1}\right\|}

\newtheorem{theorem}{Theorem}[section]

\newtheorem{lemma}[theorem]{Lemma}
\newtheorem{corollary}[theorem]{Corollary}

\theoremstyle{definition}
\newtheorem{remark}[theorem]{Remark}
\newtheorem{definition}[theorem]{Definition}

\newtheorem{prob*}[theorem]{Problem}
\numberwithin{equation}{section}
\newcommand{\invisible}[1]{}

\newcommand{\mr}{%
  \,\raisebox{-.127ex}{\reflectbox{\rotatebox[origin=br]{-90}{$\lnot$}}}\,%
}

\def\XXint#1#2#3{{\setbox0=\hbox{$#1{#2#3}{\int}$ }
\vcenter{\hbox{$#2#3$ }}\kern-.6\wd0}}

\renewcommand{\dh}{\,d\mathcal{H}^{m-1}}
\newcommand{\dx}{\,dx}

\newcommand{\barint}{
\rule[.036in]{.12in}{.009in}\kern-.16in \displaystyle\int }

\newcommand{\barcal}{\mbox{$ \rule[.036in]{.11in}{.007in}\kern-.128in\int $}}

\def\mvint_#1{\mathchoice
          {\mathop{\vrule width 6pt height 3 pt depth -2.5pt
                  \kern -8pt \intop}\nolimits_{\kern -3pt #1}}%
          {\mathop{\vrule width 5pt height 3 pt depth -2.6pt
                  \kern -6pt \intop}\nolimits_{#1}}%
          {\mathop{\vrule width 5pt height 3 pt depth -2.6pt
                  \kern -6pt \intop}\nolimits_{#1}}%
          {\mathop{\vrule width 5pt height 3 pt depth -2.6pt
                  \kern -6pt \intop}\nolimits_{#1}}}

\numberwithin{theorem}{section} \numberwithin{equation}{section}

\begin{document}

\sloppy
\begin{abstract}
We prove full boundary regularity for minimizing biharmonic maps with smooth Dirichlet boundary conditions. Our result, similarly as in the case of harmonic maps, is based on the nonexistence of nonconstant boundary tangent maps. With the help of recently derivated boundary monotonicity formula for minimizing biharmonic maps by Altuntas we prove compactness at the boundary following Scheven's interior argument. Then we combine those results with the conditional partial boundary regularity result for stationary biharmonic maps by Gong--Lamm--Wang. 
\end{abstract}

\maketitle

\section{Introduction}
In this article we study the boundary regularity of \emph{minimizing} biharmonic maps. Let us first briefly discuss the known boundary regularity results in the case of second order problems: Boundary regularity for \emph{minimizing} harmonic maps with sufficiently smooth Dirichlet boundary conditions was proved by Schoen and Uhlenbeck \cite{SU2} and for \emph{minimizing} $p$-harmonic maps by Hardt and Lin \cite{HLp} (see also \cite{FuchsIII}). The boundary regularity results for \emph{minimizing} harmonic and $p$-harmonic maps crucially depend on the existence of a monotonicity formula at the boundary. Such a formula is obtained by reflecting a comparison map used in the proof of a monotonicity formula for minimizing maps, see \cite[Lemma 1.3]{SU2}. Full boundary regularity is a consequence of the absence of nonconstant boundary tangent maps for the class of minimizing maps.

There is also a conditional result for \emph{stationary} harmonic maps \cite{Wang-partial}, which under the assumption of a boundary monotonicity formula for stationary maps yields a partial regularity at the boundary. See also \cite{Scheven-variationally} for a boundary regularity result for another class of harmonic maps. The main reason for which no \emph{unconditional} partial boundary regularity result is known for \emph{stationary} harmonic maps is the lack of a boundary monotonicity formula. Here, we would also like to point out that a boundary monotonicity formula may be obtained for \emph{all} sufficiently smooth harmonic maps. According to \cite{Lin-blowup} such a formula was obtained first by W.Y.~Ding, see also \cite{Chen-Lin} and references therein.

Now, we introduce the setting. 

Let $\n$ be a smooth, compact Riemannian manifold without boundary of dimension $n$. By Nash's embedding theorem \cite{Nash}, we may assume that $\n$ is isometrically embedded in some Euclidean space $\R^\ell$ for $\ell$ sufficiently large. For a smooth, bounded domain $\Omega\subset\R^m$, $k\in \N$, and $1\le p\le\infty$ we define the Sobolev spaces
\[
 W^{k,p}(\Omega, \n) = \left\{ u\in W^{k,p}(\Omega, \R^\ell)\colon u(x)\in\n \text{ for a.e. } x\in\Omega\right\},
\]
equipped with the topology inherited from the topology of the linear Sobolev space $W^{k,p}(\Omega, \R^\ell)$.

We define the Hessian energy (or extrinsic biharmonic energy) for $u\in W^{2,2}(\Omega, \n)$ as
\begin{equation}\label{de:bil}
 H(u)=\int_\Omega |\Delta u|^2 \, dx,
\end{equation}
where $\Delta$ is the standard Laplace operator on $\R^m$. This energy depends on the embedding $\n\hookrightarrow\R^\ell$.

A map $u\in W^{2,2}(\Omega,\n)$ is said to be \emph{weakly biharmonic} if it is a critical point (with respect to the variations in the range) of the biharmonic energy, i.e., if it satisfies
\begin{equation}\label{eq:weaklybiharmonic}
 \frac{d}{dt}\bigg\rvert_{t=0}H(\pi_\n(u+t\zeta))=0,
\end{equation}
where $\zeta\in C^\infty_0(\Omega,\mathbb{R}^{\ell})$ and $\pi_\n:\mathcal O(\n)\rightarrow \n$ denotes the nearest point projection of a neighborhood $\mathcal O(\n)\subset\R^\ell$ of $\n$ onto $\n$. 

Geometrically, biharmonic maps are solutions $u\in W^{2,2}(\Omega,\n)$ to
\begin{equation}\label{eq:elgeneral1}
 \Delta^2 u \perp T_u\n.
\end{equation}
For a derivation of the Euler--Lagrange equation of biharmonic maps see \cite{Wang2004}.

We say that a map  $u~\in~W^{2,2}(\Omega, \n)$ is \emph{stationary biharmonic} if in addition to \eqref{eq:weaklybiharmonic} it is a critical point with respect to all variations of the domain, i.e., $u$ satisfies 
 \begin{equation}\label{eq:definitonstationaryharmonic}
  \frac{d}{dt}\bigg\rvert_{t=0}H(u(\cdot + t\xi(\cdot)))=0
 \end{equation}
whenever $\xi\in C_c^\infty(\Omega,\R^m)$.

In this paper we will be focused on a subclass of biharmonic maps, called \emph{minimizing biharmonic} maps, which are maps $u\in W^{2,2}(\Omega,\n)$ satisfying
\[
 H(u)\le H(v)
\]
for all $v\in W^{2,2}(\Omega,\n)$ such that $u-v\in W^{2,2}_0(\Omega,\R^\ell)$.

The study of regularity of biharmonic maps was initiated by Chang et al. in \cite{CWY}. They investigated mappings with values in the sphere $\S^{\ell-1}$. In case $m=4$ they proved the regularity of \emph{all} biharmonic maps, while for $m\ge5$ they proved that stationary biharmonic maps are $C^\infty$ except a closed set $\Sigma$ of Hausdorff dimension at most $m-4$. Their result was partially extended to general target manifolds by C.~Wang in \cite{Wangbiharmonicremarks,Wangbiharmonicintoriem,Wang2004}. Alternative proofs were given by Strzelecki \cite{Strzeleckibiharmonic} for $m=4$, $\n=\S^{\ell-1}$, Lamm with Rivi\`{e}re \cite{LR} for $m=4$ and arbitrary $\n$, and Struwe \cite{Struwe2008} for $m\ge 5$ and an arbitrary target manifold $\n$. 

In \cite{CWY} Chang, L.~Wang and Yang derived from the stationary assumption a monotonicity formula, although only for sufficiently regular maps. That formula was crucial in the proof of partial regularity for $m\ge5$. A rigorous proof of the monotonicity formula was given by Angelsberg in \cite{Ang}.

In the case of minimizing biharmonic maps the partial regularity results may be strengthened. First it was observed by Hong and C.~Wang in \cite{HW} that for $\n=\S^{\ell-1}$ the singular set $\Sigma$ has Hausdorff dimension at most $m-5$. One can prove the optimality of this result considering a map $\frac{x}{|x|}:B^5\rightarrow\S^4$ (see \cite[Proposition A1.]{HW}). Finally, Scheven in \cite{Scheven} reduced the dimension of singular set of minimizing mappings to an arbitrary target manifold $\n$. His result states that, as in the case $\n=\S^{\ell-1}$, the singular set $\Sigma$ of minimizing biharmonic maps has $\dim_\h \Sigma\le m-5$. 

In a recent paper Breiner and Lamm \cite{BL} prove that each minimizing biharmonic map is \emph{locally} in $W^{4,p}$ for $1\le p \le 5/4$. 

Let us mention here two inconclusive results in the direction of boundary regularity. Firstly, it was shown in \cite{LammWang} by Lamm and C.~Wang that polyharmonic maps, in the conformal case $m=2k$, enjoy the property of being continuous in a neighborhood of the boundary. The proof is strongly dependent on the relation $m=2k$ and one might not extend this method to the case $m>2k$. The other result concerns partial boundary regularity for stationary maps. It was shown in \cite{GLW} by Gong et al. that if we impose an additional condition on the boundary mapping then there exists a closed subset $\Sigma\subseteq\overline{\Omega}$, with $\mathcal{H}^{m-4} (\Sigma)=0$ such that the stationary biharmonic map is smooth up to the boundary, except possibly the set $\Sigma$. The additional condition is the boundary monotonicity formula. Unlike the interior monotonicity formula, the boundary monotonicity formula is an artificial assumption --- it is unknown whether it can be deduced for all stationary maps. The result \cite{GLW} is a biharmonic counterpart of a result by C.~Wang \cite{Wang-partial} for stationary harmonic maps. 

We are interested in the boundary regularity of minimizing biharmonic maps. We assume that $u$ satisfies the Dirichlet boundary condition. 
More precisely, let $\varphi\in C^\infty(\Omega_\delta,\n)$ be given for a $\delta>0$, where 
\[\Omega_\delta=\{x\in\overline{\Omega}:\dist(x,\partial \Omega)<\delta\}.\] Then $u$ satisfies
 \begin{equation}\label{eq:dirichletboundary}
  \brac{u,\frac{\partial u}{\partial \nu}}\Bigg\rvert_{\partial\Omega}=\brac{\vp,\frac{\partial \vp}{\partial \nu}}\Bigg\rvert_{\partial\Omega},
 \end{equation}
 where $\nu$ denotes the outer normal vector. 
 
Similarly as in the case of harmonic maps a boundary monotonicity formula may be proved for sufficiently smooth biharmonic maps. Gong, Lamm, and C.~Wang \cite{GLW} proved that \emph{all} biharmonic maps that are in $W^{4,2}$ satisfy a boundary monotonicity formula. Recently, Altuntas derived the boundary monotonicity formula for \emph{minimizing} biharmonic maps \cite{Serdar}.
 
{\bf Statement of result.}
We show that the conditional partial regularity result of Gong et al. can be strengthen to \emph{unconditional} full boundary regularity in the case of \emph{minimizing} biharmonic maps. 
 
 \begin{theorem}\label{thm:mainbih}
 Let $m\ge 5$, $\vp\in C^\infty(\Omega_\delta,\n)$ for some $\delta>0$, assume that $u\in W^{2,2}(\Omega,\n)$ is a minimizing biharmonic map, which satisfies the Dirichlet boundary conditions for a $\vp \in C^\infty(\Omega_\delta, \n)$ in the sense of \eqref{eq:dirichletboundary}. Then, $u$ is smooth on a full neighborhood of the boundary $\partial\Omega$.
 \end{theorem}
 
Similarly as in the case of harmonic \cite{SU2} and $p$-harmonic \cite{HLp} maps the complete boundary regularity is based on the nonexistence of nonconstant boundary tangent maps. We will consider tangent maps at the boundary and prove that they arise as strong limits of rescaled maps on some smaller domain, containing a portion of the boundary. In order to obtain a strong convergence from a sequence we initially only know is uniformly bounded in $W^{2,2}$ we will prove an analogue of Scheven's compactness result. 

Scheven, following the result for harmonic maps \cite{Lin-blowup}, has based his argument on an analysis of defect measures. We follow his general strategy, modifying numerous technical details so that the proof works for a map obtained via a higher order reflection across a flat portion of the boundary. 

We will not prove that a limit $u$ of a weakly convergent sequence of minimizing maps 
$(u_j)_{j\in\Z}$ is again minimizing. Such a result, is known only in the case when $\n=\S^{\ell-1}$ (see \cite[Lemma 3.3.]{HW}). In the case of harmonic maps, such a result is known for minimizing maps into arbitrary target manifolds. Since the maps $u_j$ and $u$ slightly differ on the boundary one may not use directly the definition of minimizing map to compare their energies. A tool for comparing those energies was provided by Luckhaus and his lemma in \cite{Lu}. Unfortunately we may not use directly Luckhaus’s lemma to maps from $W^{2,2}$. An analogue of this lemma is not known in the biharmonic setting. 

Instead, similarly as in \cite{SU1, SU2} and \cite{HLp}, for us it will be sufficient that in very simple situations a limit of \emph{minimizing} maps is again \emph{minimizing}. By a repeated formation of tangent boundary maps we arrive at a boundary tangent map which has a special form --- it is independent of the first $(m-5)$-variables, homogeneous of degree 0, whose only discontinuity may occur at the origin. It was proved by Scheven that such maps are in fact minimizing (cf. Lemma \ref{le:mintanmaps}). 

{\bf Outline of the article.}
The paper is organized as follows. In Section \ref{s:factsbiharmonic} we state various facts on biharmonic maps which will be needed in the following proofs. In Section \ref{s:compactness} we give a boundary analogue of Scheven's compactness result for minimizing biharmonic maps. In Section \ref{s:tangentmaps} we focus on the tangent maps at the boundary. We prove that there exist no nonconstant boundary tangent maps and finally give the proof of the main result.
 
 {\bf Notation.} We use the following notation
\[
 \int_{\partial B_r\setminus\partial B_\rho} f \dh := \int_{\partial B_r} f \dh - \int_{\partial B_\rho} f \dh. 
\]
 For balls centered at the origin we often write $B_r(0)=B_r$, for $B_1$ we simply write $B$. Sometimes to emphasize the dimension of a ball we will write $B^k$, for a $k$-dimensional ball. We also write $\R^m_+ = \{x=(x_1,\ldots,x_m)\in\R^m: x_m>0\}$, $\R^m_- = \{x=(x_1,\ldots,x_m)\in\R^m: x_m<0\}$, $B^+_r(a)= B_r(a)\cap \R^m_+$, and $B^-_r(a) = B_r(a)\cap \R^m_-$. For the the flat part of the boundary of $\partial B_\sigma^+$ we use
\[
 T_\sigma = \{x\in B_\sigma\colon x_m=0\}.
\]

In what follows we will use sequences and partial derivatives, for partial derivatives we write
\[
 \frac{\partial}{\partial x_i} u = \partial_i u = u_{x_i},
\]
while $u_i$ will denote the $i$-th element of a sequence of maps $(u_j)_{j\in\N}$. For simplicity we will try to use the following convention: Letters $u,\ v,\ w$ will be used to denote maps from $B^+$ into $\n$, whereas $\w{u},\ \w{v}, \w{w}$ will denote maps from $B$ into $\R^\ell$. The constant $C$ traditionally stands for a general constant and may vary from line to line. 
 
\section{Facts on regularity of biharmonic maps}\label{s:factsbiharmonic}
In this section we gather facts from the regularity theory of biharmonic maps, which will be needed later on. We begin by recalling the definition of Morrey spaces, for more details see, e.g., \cite[Chapter 3]{Giaquinta1993}.

Let $p\ge 1$, $\lambda>0$, and $\Omega$ be a bounded domain in $\R^m$. We say that a function $f\in L^p(\Omega)$ belongs to the Morrey space $L^{p,\lambda}(\Omega)$ if 
\begin{equation}\label{eq:morrey}
 \norm{f}_{L^{p,\lambda}(\Omega)}^p := \sup_{a\in\Omega,\, r>0} r^{-\lambda}\int_{B_r(a)\cap\Omega} |f(x)|^{p}\, dx < \infty.
\end{equation}

The following boundary decay estimate for biharmonic maps that satisfy a smallness condition in Morrey norm is due to Gong, Lamm, and C.~Wang, see \cite[Lemma 3.1, p. 179]{GLW}. 
\begin{lemma}\label{le:eps-reg}
 There exists $\ve>0$ and $\theta\in\big(0,\frac12\big)$ such that if $u\in W^{2,2}(B^+,\n)$ is a biharmonic map satisfying 
 
 \begin{equation}
 u \Big\rvert_{T_1} = \vp \Big\rvert_{T_1} \quad\text {and}\quad \frac{\partial u}{\partial x_m}\bigg\rvert_{T_1} = \frac{\partial \vp}{\partial x_m}\bigg\rvert_{T_1} \quad\text{for some}\quad \vp\in C^\infty(\overline{B^+},\n) 
 \end{equation}
and
\begin{equation}\label{eq:condepsreg}
 \norm{\nabla^2 u}^2_{L^{2,m-4}(B^+)} + \norm{\nabla u}^4_{L^{4,m-4}(B^+)}\le \ve,
\end{equation}
then
\begin{equation}
\norm{\nabla u}_{L^{2,m-2}(B^+_{\theta})} \le \frac12 \norm{\nabla u}_{L^{2,m-2}(B^+)} + C\theta\norm{\nabla\vp}_{C^1(B^+)}.
\end{equation}
In particular, $u\in C^\infty\left(\overline{B^+_{\frac12}},\n\right)$.
\end{lemma}

The following theorem is a key-ingredient in the regularity theory. It was first proved for sufficiently regular maps by Chang, L.~Wang, and Yang in \cite[Proposition 3.2.]{CWY} and for the general case by Angelsberg in \cite{Ang}. We employ the notation for $\Phi_u$ from \cite[Theorem 2.3]{Scheven}.

\begin{theorem}[Monotonicity formula]\label{monoform}
 Let $u\in W^{2,2}(B_R^+,\n)$ be stationary biharmonic and $a\in B_{R/4}$. Then the expression
 \begin{align*}
  &\Phi_u(a,r) := \\ & \quad r^{4-m}\int_{B_r(a)}|\Delta u|^2 \, dx \\
                &\quad + 2\int_{\partial B_r(a)}\left(\frac{(x^i-a^i)u_{x_j}u_{ x_i x_j}}{|x-a|^{m-3}} - 2\frac{\left((x^i - a^i)u_{x_i}\right)^2}{|x-a|^{m-1}} +2\frac{|\nabla u|^2}{|x-a|^{m-3}}\right) \, d\h^{m-1}  
 \end{align*}
is well defined for a.e. $0<r\le R/4$ and monotonously nondecreasing for all $r$ outside a set of measure zero. more precisely, there holds for a.e. $0<\rho<r\le R/4$
\begin{equation}\label{mono}
\begin{split}
 \Phi_u(a,r&) - \Phi_u(a,\rho) =\\
 &4\int_{B_r(a)\setminus B_\rho(a)}\left(\frac{\left(u_{x_j}+(x^i-a^i)u_{ x_i x_j}\right)^2}{|x-a|^{m-2}}+(m-2)\frac{\left((x^i-a^i)u_{x_i}\right)^2}{|x-a|^m}\right) \, dx.
 \end{split}
\end{equation}

\end{theorem}

\begin{remark}
The Angelsberg's proof of monotonicity formula, roughly speaking, is based on inserting a correct test function in the so-called first variational formula (an equation which follows from the definitions of stationary harmonic maps). This idea follows the proof of monotonicity formula for stationary \emph{harmonic} maps (see, e.g., \cite{Simon1996}), which in turn is based on the proof of the monotonicity formula for Yang--Mills fields, see \cite{Price1983}. The first publication of a monotonicity formula for \emph{minimizing} harmonic maps seems to be \cite[Proposition 2.4.]{SU1}, which by reflection arguments can be extended to boundary monotonicity formula for minimizing harmonic maps, see \cite[Lemma 1.3.]{SU2}. The proof in \cite{SU1} proof relies on constructing a comparison map. It would be interesting to obtain an analogous proof for \emph{minimizing} biharmonic maps. 
\end{remark}

The following theorem is a boundary analogue of the monotonicity formula. It was first proved for any $W^{4,2}$ biharmonic map (not necessary minimizing), see \cite[Section 2]{GLW}. Recently a boundary monotonicity formula was derived for all \emph{minimizing} biharmonic mappings in $W^{2,2}$ with sufficiently smooth boundary data, see \cite{Serdar}.

\begin{theorem}\label{thm:boundarymono}
Let $u\in W^{2,2}(\Omega,\n)$ be a minimizing biharmonic map with a boundary map $\vp\in C^\infty(\Omega_\delta,\n)$ as in \eqref{eq:dirichletboundary}. Then $u$ satisfies the boundary monotonicity inequality, i.e., there exist $R_0>0$ and $C = C(m,\partial\Omega,\delta, \|\vp\|_{C^4(\Omega_\delta)})$ such that for any $a\in\partial\Omega$ and $0<\rho\le r\le R_0$, there holds

\begin{equation}\label{eq:boundarymono}
\begin{split}
H^+_{u}(a,\rho) + e^{C\rho}R^+_{u}&(a,\rho) + P^+_{u}(a,\rho,r)\\
& \le e^{Cr}H^+_{u}(a,r)+ e^{Cr}R^+_{u}(a,r) +Cre^{Cr},
\end{split}
\end{equation}
where
\begin{align}
H^+_{u}(a,\tau) &:=\tau^{4-m} \int\limits_{B_{\tau}(a)\cap\Omega}|\nabla^2(u-\vp)|^2 \, dx,\label{de:Hplus}\\
P^+_{u }(a,\rho,r)&:=\int_{\brac{B_r(a)\setminus B_\rho(a)}\cap\Omega}\frac{|(u-\vp)_{ x_j} +(x-a)^i(u-\vp)_{ x_i x_j}|^2}{|x-a |^{m-2}} \, dx\nonumber\\
& \quad + (m-2)\int_{\brac{B_r(a)\setminus B_\rho(a)}\cap\Omega}\frac{|(x-a)^i(u-\vp)_{ x_i}|^2}{|x-a|^{m}} \, dx \label{de:Pplus} 
\end{align}
and 
\begin{equation}\label{de:Rplus}
R^+_{u }(a,\tau) = (F^+_{u }(a,\tau) + G^+_{u }(a,\tau)) 
\end{equation}
for
\[
 \begin{split}
  F^+_{u }(a,\tau)&:= 2\tau^{3-m}\int\limits_{\partial B_\tau(a)\cap\Omega}(x-a)^i(u-\vp)_{ x_j}(u-\vp)_{ x_i x_j} \, d\h^{m-1}\\
  &\quad-4\tau^{3-m}\int\limits_{\partial B_\tau(a)\cap\Omega}\brac{\frac{|(x-a)^i(u-\vp)_{ x_i}|^2}{|x-a|^{2}}-|\nabla(u-\vp)|^2}\, d\h^{m-1},\\
  G^+_{u }(a,\tau)&:=2\tau^{4-m}\int\limits_{\partial B_\tau(a)\cap\Omega}\Bigg( \left\langle\Delta(u-\vp),\frac{\partial}{\partial r}(u-\vp)\right\rangle \\ 
  &\quad \quad \quad \quad \quad \quad \quad \quad \quad- \left\langle\nabla(u-\vp),\frac{\partial}{\partial r}(\nabla(u-\vp))\right\rangle \Bigg)\, d\h^{m-1}.
 \end{split}
\]
In the latter $\frac{\partial}{\partial r}$ is the directional derivative in the direction of the outward pointing unit normal for $\partial B_\tau(a)$.
\end{theorem}

The following result is a consequence of the boundary monotonicity formula and for the proof we refer to the appendix \ref{a:biharmonic}.
\begin{lemma}\label{le:boundednessinmorrey}
 Let $u\in W^{2,2}(B^+,\n)$ be a minimizing biharmonic map with boundary value $\vp$ as in \eqref{eq:dirichletboundary} and let additionally $\norm{u-\vp}_{W^{2,2}(B^+)}<\infty$. Then, for some $\Lambda>0$ we have $\norm{\nabla^2(u-\vp)}_{L^{2,m-4}(B^+)}<\Lambda$.
\end{lemma}

The following is also a consequence of the boundary monotonicity formula, the proof can be found in \cite[Lemma 4.1]{GLW}. (Compare also in the interior case in \cite[Lemma 4.8]{CWY}, \cite[Lemma 5.3]{Wang2004}, \cite[Appendix B]{Struwe2008}).

\begin{lemma}\label{le:glw41}
Assume that the hypothesis of Theorem \ref{thm:mainbih} is fulfilled. Then, there exist $\eps_1>0,\, \theta\in(0,1)$, $C_1 =C_1(m,\Omega,\n)$, and $R_1=R_1(R_0,\eps_1)$ such that if for $a\in\partial\Omega$ and $0<\tau\le R_1$
\begin{equation}\label{eq:regcond}
 \tau^{4-m}\int_{ B_\tau(a)\cap\Omega}\brac{|\nabla^2 u|^2 + \tau^{-2}|\nabla u|^2}\dx \le \eps_1^2,
\end{equation}
then
\begin{equation}\label{eq:oszacowaniele:glw41}
\sup_{B_\rho(b)\subset (B_{\theta \tau}(a)\cap\, \Omega)} \rho^{4-m}\int_{B_\rho(b)\cap\,\Omega}(|\nabla^2 u|^2 + |\nabla u|^4)\dx\le C_1\eps_1.
\end{equation}
\end{lemma}

The following epsilon regularity result is the main result of \cite{GLW}.
\begin{theorem}\label{thm:condbdrregstat}

 Let $m\ge 5$, $\vp\in C^\infty(\Omega_\delta,\n)$ for some $\delta>0$, assume that $u\in W^{2,2}(\Omega,\n)$ is a minimizing biharmonic map, which satisfies the boundary monotonicity inequality \eqref{eq:boundarymono}. Then, there exists an $\eps_2>0$ such that $u\in C^{\infty}(\overline{\Omega}\setminus\Sigma)$, where the singular set is given by 
 \[
  \Sigma:= \left\{a\in\overline{\Omega}: \liminf_{r\searrow0} r^{4-m}\int_{B_r(a)\cap\Omega}\brac{|\nabla^2u|^2 + |\nabla u|^4}\dx\ge\eps_2 \right\}
 \]
and $\h^{m-4}(\Sigma)=0$.
\end{theorem}

\begin{lemma}\label{le:odpSchThm26}
 There are constants $\eps_3>0$ and $\theta\in(0,1)$ such that under the assumptions of Theorem \ref{thm:mainbih} a minimizing biharmonic map $u\in C^{\infty}(B_r^+,\n)$ with boundary values $\vp$ as in \ref{eq:dirichletboundary} with
 \[
  r^{4-m}\int_{B_r^+(a)}\brac{|\nabla^2 u|^2 + r^{-2}|\nabla u|^2}\dx \le \eps_3
 \]
satisfies $\norm{\nabla u}_{C^2(B_{\theta r}(a)\cap\Omega)}\le 1$.
\end{lemma}

\begin{proof}
The proof follows \cite[proof of Theorem 2.6]{Scheven}. We list the following boundary analogues needed to replace the interior facts used in \cite{Scheven}:
\begin{itemize}
 \item Lemma 2.4 (i) in \cite{Scheven} by Lemma \ref{le:boundednessinmorrey};
 \item Lemma 2.4 (ii) in \cite{Scheven} by Lemma \ref{le:glw41};
 \item Theorem 2.5 in \cite{Scheven} by Theorem \ref{thm:condbdrregstat};
 \item Theorem 3.1 from \cite{Moserunpublished} by Theorem 4.1 from \cite{Moserunpublished} 
\end{itemize}
\end{proof}

Similarly as in \cite[Corollary 2.7]{Scheven}, we obtain the following consequence.
\begin{corollary}\label{co:odpschcor27}
 Let $\eps_0:= \min(\eps_1^2,C_{1}\eps_1,\eps_2,\eps_3)$ with constants introduced in Lemma \ref{le:glw41}, Theorem \ref{thm:condbdrregstat} and Lemma \ref{le:odpSchThm26}. Then there exists $\theta\in(0,1)$ such that for any minimizing biharmonic map $u\in W^{2,2}(B^+_r(a),\n)$, the estimate
 \begin{equation}\label{eq:odpcor27}
  r^{4-m}\int_{B_r^+(a)}\brac{|\nabla^2 u|^2 + r^{-2}|\nabla u|^2}\dx \le \eps_0
 \end{equation}
implies $u\in C^3(B_{\theta r}(a)\cap\Omega,\n)$ and $\norm{u}_{C^3(B_{\theta r}(a)\cap\Omega,\n)}$ is bounded by a constant depending only on $\n$.
\end{corollary}

\section{Compactness at the boundary}\label{s:compactness}
For simplicity we will assume that $\Omega = B^+_4$. In this situation our boundary condition states that 
 \begin{equation}\label{eq:boundarydataonhalfball}
  \brac{u,\frac{\partial u}{\partial x_m}}\Bigg\rvert_{T_4}=\brac{\vp,\frac{\partial \vp}{\partial x_m}}\Bigg\rvert_{T_4}.
 \end{equation}

The following compactness theorem is due to Scheven, cf. \cite[Theorem 1.5.]{Scheven}. Here we present a boundary analogue of this statement.
\begin{theorem}\label{thm:comp}
There is a constant $C_\vp=C_\vp(m)$ with the following property. 

Let $M(B^+_{4})\subseteq W^{2,2}(B^+_{4},\n)$ be the closure with respect to the $W^{2,2}_{loc}$-topology of the set of minimizing biharmonic maps. Let $u_i\in M(B^+_{4})$, be a sequence of maps with boundary values $\varphi_i$ in the sense of \eqref{eq:dirichletboundary} and $\vp\in C^{\infty}(B^+_4)$. Moreover, let 
\begin{equation}\label{eq:compthmassumptions}
\begin{split}
&\vp_i\rightarrow \vp \quad \text{ strongly in }W^{2,2}_{loc} \text { and } L^6_{loc},\\
&\sup_{i\in\N}\norm{\vp_i}_{C^2(B_4^+)}<\frac{\eps_0}{2},\quad \sup_{i\in\N}\norm{\vp_i}_{C^3(B^+_4)}<C(\n),\\
&\sup_{i\in\N} \norm{u_i}_{W^{2,2}(B_4^+)}<\infty, \ \text{ and }\quad \int_{B^+_4}|\Delta \vp|^2\dx<C_\vp,  
\end{split}
\end{equation}
where $\eps_0$ is the constant from Corollary $\ref{co:odpschcor27}$. Then, there is a map $u\in M(B^+_{4})$ such that, up to a subsequence, $u_{i}\rightarrow u$ in $W^{2,2}(B^+_{1/2},\n)$ as $i\rightarrow\infty$.
\end{theorem}

\begin{remark}
 In fact the $L^6$ convergence of $\nabla \vp_i$ can be relaxed to $L^{4+\epsilon}$ for any $\epsilon>0$. For this purpose, in the proof of Theorem \ref{thm:comp}, one should replace Young's inequality with exponents 3 and $\frac32$ in the estimate \eqref{eq:vtildawithouttsigma}, by 
 \[
  |\nabla \vp_i|^2|\nabla u_i|^2 \le \frac{2|\nabla \vp_i|^{4+\epsilon}}{4+\epsilon} + \frac{(2+\epsilon)|\nabla \wu_i|^{\frac{8+2\epsilon}{2+\epsilon}}}{4+\epsilon}.
 \]
To proceed with the proof, the only important thing for us in this estimate is that the exponent at $|\nabla \wu_i|$ stays below 4.
\end{remark}

We will extend $(u-\vp)$ onto the whole ball by a higher order reflection, for properties of the reflection see, e.g, \cite[proof of Theorem 4.26]{Adams}. We choose such a reflection, which preserves $C^3$ continuity of a map. Let $u\in W^{2,2}(B_4^+,\n)$ with boundary values $\vp$ as in \eqref{eq:dirichletboundary}, then the reflection $\w{u}$ is given by
\begin{equation}\label{def:reflection}
\begin{aligned}
 \widetilde{u}(x)= \begin{cases} u(x)-\vp(x) &\mbox{for } x_m\ge 0, \\
\sum_{i=1}^4 \lambda_i (u-\vp)\left(x',-\frac{x_m}{i}\right) & \mbox{for } x_m < 0, \end{cases} 
\end{aligned}
\end{equation}
where $x'=(x_1,\ldots,x_{m-1})$ denotes the first $(m-1)$-coordinates and the constants $\lambda_i$ are determined by the system
\[
\sum_{i=1}^4 \lambda_i \left(-\frac{1}{i}\right)^j=1\quad\text{ for } j=0,1,2,3. 
\]
Here $\lambda_1 = -10$, $\lambda_2 = 160$, $\lambda_3 = -405$, and $\lambda_4 = 256$. We note that $\wu$ in general does not have values in $\n$. Next, observe that since $u-\vp \in W^{2,2}_0(B^+_4,\R^\ell)$ we have $\w{u}\in W^{2,2}(B_4,\R^\ell)$. Let $\alpha=(\alpha_1,\ldots,\alpha_m)$ be such that $|\alpha|\le3$ and 
\[
 E_\alpha \wu (x) = \begin{cases}
                   u(x) - \vp(x)&\mbox{for } x_m\ge 0, \\
                   \sum_{i=1}^4 \lambda_i\brac{-\frac1i}^{\alpha_m} (u-\vp)\left(x',-\frac{x_m}{i}\right) & \mbox{for } x_m < 0.
                  \end{cases}
\]
It follows easily that, if $u-\vp\in C^3(\overline{B}_4^+,\R^\ell)$, then $\w{u}\in C^3(B_4,\R^\ell)$ and \[D^{\alpha}\wu (x) = E_\alpha D^\alpha \wu(x).\]
Moreover, if $a\in T_4$ and $r>0$ is such that $B_r(a)\subset B_4$, then for $x_m\in B_r^-(a)$ we have $-\frac{x_m}{i}\in B_r^+(a)$. Thus, for $|\alpha|=2$, we have the following estimate
\begin{equation}\label{eq:reflectionnormestimate}
\begin{split}
 \int_{B_r(a)} |\nabla^2 \w{u}|^2 \dx  &= \int_{B_r^+(a)} |\nabla^2 (u-\vp)|^2 \dx  \\
 &\quad + \int_{B_r^-(a)} \left|\sum_{i=1}^4 \lambda_i\brac{-\frac1i}^{\alpha_m} D^\alpha(u-\vp)\left(x',-\frac{x_m}{i}\right)\right|^2\dx\\
 &\le C_{ref}\int_{B_r^+(a)}|\nabla^2 (u-\vp)|^2\dx
\end{split}
 \end{equation}
with $C_{ref} \le832$. 
 
For reflected maps as in \eqref{def:reflection}, in order to prove Theorem \ref{thm:comp} we closely follow Scheven's Section 3.1 of \cite{Scheven}, adjusting numerous technical details whenever necessary. All of the tools used by Scheven in proofs of Lemmata, Theorem and Corollaries have their boundary analogues, therefore, we will be rather brief in most of the proofs below. The difference here is that instead of working with minimizing maps themselves defined on half balls, we will work with higher order reflections of the differences of the mappings and their boundary data. Moreover, the boundary monotonicity formula has a little bit different form from the interior one and yields an additional term (which can still be well controlled).

We shall work with the following definition of convergence of pairs of sequences of maps and measures (slightly different from the one used in \cite{Scheven}).
\begin{definition}
For $i\in\N$, let $u_i\in W^{2,2}(\Omega,\n)$ and $\nu_i$ be Radon measures on $\overline{\Omega}$. We abbreviate $\mu_i:=(u_i,\nu_i)$. For a map $u_0\in W^{2,2}(\Omega,\n)$, a Radon measure $\nu_0$ on $\overline{\Omega}$ and $\mu_0:=(u_0,\nu_0)$ we write $\mu_i\rightrightarrows\mu_0$ as $i\rightarrow\infty$ if and only if 
 \begin{align*}
u_i &\rightharpoonup u_0  & &\text{ weakly in } W^{2,2}(\Omega)\\
u_i &\rightarrow u_0 & &\text{ strongly in } W^{1,2}(\Omega) \text{ and for a.e. }x\in\Omega \\
|\nabla^2 u_i|^2 dx + \nu_i &\rightharpoonup  |\nabla^2 u_0|^2dx + \nu_0 & &\text{ in the sense of measures.}
\end{align*}
\end{definition}
We recall that, by Lemma \ref{le:boundednessinmorrey}, any sequence of minimizing biharmonic maps $u_i$ with boundary conditions $\vp_i$ as in \eqref{eq:dirichletboundary}, such that $\sup_{i\in\N}\norm{\w{u}_i}_{W^{2,2}(B_4^+)}<\infty$, where $\w{u}_i$ is the reflection given by \eqref{def:reflection}, satisfies also $\sup_{i\in\N}\norm{\w{u}_i}_{L^{2,m-4}}(B_4^+)<\Lambda$ for some $\Lambda>0$.

We modify Scheven's set $\mathcal B^{M}_\Lambda$ to our purposes and let
\begin{equation}\label{eq:blambda}
\w{\mathcal{B}}_\Lambda : = 
 \left\{ \begin{array}{l|l}
    &(\w{u_i},0)\rightrightarrows(\w{u},\w{\nu}), \text{ where } \w{u_i}\in W^{2,2}(B_4,\R^\ell), \\
    &\w{u_i} \text{ are the reflections given in } \eqref{def:reflection}\\ 
    &\text{of minimizing biharmonic maps } \\
    (\w{u},\w{\nu})\, & u_i\in W^{2,2}(B^+,\n)\text{ with boundary values } \vp_i \\
    &\text{as in }\eqref{eq:dirichletboundary}\text{, satisfying boundary }\\ &\text{monotonicity formula } \eqref{eq:boundarymono},\\
    &\text{assumptions } \eqref{eq:compthmassumptions},\text{ and }\norm{\nabla^2 \w{u_i} }^2_{L^{2,m-4}(B)}\le\Lambda
  \end{array} \right\}.
\end{equation}

Combining boundary monotonicity \eqref{eq:boundarymono} with \eqref{eq:reflectionnormestimate} we obtain 

\begin{equation}\label{eq:monorefl}
\begin{split}
\rho^{4-m}&\int_{B_\rho(a)}|\nabla^2 \wu|^2 \dx + Ce^{C\rho} R^+_u(a,\rho)\\
&\le C\brac{\rho^{4-m}\int_{B_\rho^+(a)}|\nabla^2 (u-\vp)|^2 \dx + e^{C\rho} R^+_u(a,\rho)} \\
&\le C\brac{e^{Cr}r^{4-m}\int_{B_r^+(a)}|\nabla^2 (u-\vp)|^2 \dx + e^{Cr} R^+_u(a,r) + Cre^{Cr}}\\
&\le C\brac{e^{Cr}r^{4-m}\int_{B_r(a)}|\nabla^2 \wu|^2 \dx + e^{Cr} R^+_u(a,r) + Cre^{Cr}}.
\end{split}
\end{equation}

\begin{lemma}\label{le:convergenceofmonocomp}
 Assume $\w{\mathcal{B}}_{\Lambda}\ni (\w{u}_i,0)\rightrightarrows (c,\w{\nu})$ as $i\rightarrow\infty$ for a constant $c\in\R^\ell$ and a Radon measure $\w{\nu}$ on $\overline{B}_4$. Then for each $a\in B$, there is a subsequence $\{i_k\}_{k\in\N}$ such that for a.e. $0<r<1$ we have
 \begin{equation}\label{eq:zbieznoscpsiprzystalej}
r^{4-m} \int_{B_r(a)}|\nabla^2 \wu|^2 \dx + Ce^{Cr} R^+_u(a,r)\rightarrow r^{4-m}\widetilde{\nu}(B_r(a)) \text{ as } k\rightarrow\infty
  \end{equation}
 and for every $a\in B$ and for a.e. $0<\rho\le r<1$
 \begin{equation}\label{eq:measureinequality}
  \rho^{4-m}\widetilde{\nu}(B_\rho(a))\le C r^{4-m}\widetilde{\nu}(B_r(a)) + Cre^{Cr}.
 \end{equation}
If for a minimizing sequence of biharmonic maps $\{u_i\}$ we have $\widetilde{u}_i\rightarrow\widetilde{u}_0$ strongly in $W^{2,2}(B_4,\n)$ as $i\rightarrow\infty$, then for every $a\in B$ there is a subsequence $\{i_k\}\subset\N$ such that
\begin{equation}\label{eq:psistrong}
\begin{split}
 H^+_{u_{i_k}}(a,r)&\rightarrow H^+_{u_0}(a,r) \quad \text {for a.e. } 0<r<1 \text{ as } k\rightarrow\infty;\\
 R^+_{u_{i_k}}(a,r)&\rightarrow R^+_{u_0}(a,r) \quad \text {for a.e. } 0<r<1 \text{ as } k\rightarrow\infty,
\end{split}
 \end{equation}
 where $H^+$ and $R^+$ are the quantities from the boundary monotonicity formula and are defined in \eqref{de:Hplus} and \eqref{de:Rplus} respectively.

\end{lemma}
\begin{proof}
  The proof is essentially the same as the proof of \cite[Lemma 3.2]{Scheven}. We briefly note the following differences. In order to obtain the convergence in  \eqref{eq:zbieznoscpsiprzystalej} we need to ensure that the term $Ce^{Cr}R^+_{u_i,\vp_i}(a,r)$ converges on a subsequence to $0$.
 With addition to the argument used by Scheven we let $a\in B$ be fixed  and 
 \[
  \widetilde{g}_i(\tau) := \int\limits_{\partial B_\tau(a)\cap B^+_4}\Bigg( \left\langle\Delta\wu_i,\frac{\partial}{\partial r}\wu_i\right\rangle 
  - \left\langle\nabla\wu_i,\frac{\partial}{\partial r}(\nabla\wu_i))\right\rangle \Bigg)\dh
 \]
for all $i\in\N$ and a.e $\tau\in(0,1]$. Then
\begin{equation*}
 \begin{split}
  \norm{g_i}_{L_1([0,1])}& \le C\norm{\wu_i}_{W^{2,2}}\norm{\nabla\wu_i}_{L^2}\rightarrow 0 \quad \text{as}\quad i\rightarrow\infty,
 \end{split}
\end{equation*}

which, after the same arguments as Scheven's, yields \eqref{eq:zbieznoscpsiprzystalej}. 

The inequality \eqref{eq:measureinequality} is a consequence of the monotonicity for the reflected map $\wu$ \eqref{eq:monorefl}.

In the second case, in addition to Scheven's argument, by the strong convergence we get strong convergence  in $L^1([0,1])$ of $f_i$ (defined as in \cite[proof of Lemma 3.2.]{Scheven} with $u_i$ replaced by the difference $(u_i-\vp_i)$) and of $g_i$ to some $f_0$ and $g_0$. We may choose a subsequence so that $f_{i_i} \rightarrow f_0$ a.e and $g_{i_k} \rightarrow g_0$ a.e. as $k\rightarrow\infty$. Together with the strong convergence of $u_i\rightarrow u_0$ we obtain \eqref{eq:psistrong}.

\end{proof}

We employ Scheven's definitions of rescaled pairs to our case of reflected maps. First, we observe that for every $\widetilde{\mu} = (\w{u},\w{\nu})\in \widetilde{\mathcal{B}}_\Lambda$ we have by definition 

\begin{equation}\label{eq:maloscdotangent}
 \sup_{a\in B,\, \rho<1}\rho^{4-m}\brac{\int_{B_\rho(a)}|\Delta \w{u}|^2 \ dx + \w{\nu}(B_\rho(a))}\le \Lambda.
\end{equation}

\begin{definition}\label{de:rescaled}
The tangent data of $\w{\mu}$ are defined as follows. Let $a\in B$ and $0<r<1$. For a pair $\w{\mu} = (\w{u}, \w{\nu}) \in \w{\mathcal{B}}_\Lambda$ we define the rescaled pair $\w{\mu}_{a,r} := (\w{u}_{a,r}, \w{\nu}_{a,r})$ by
\begin{align*}
 \w{u}_{a,r}(x)&:= \w{u}(a+rx) \quad & &\mbox{ for } x\in B_{1/r}(0)\\
 \w{\nu}_{a,r}(A)&:=r^{4-m}\w{\nu}(a+rA) \quad & &\mbox{ for every Borel set } A\subset B_{1/r}(0),
\end{align*}
in the first definition we have chosen some representative of $\w{u}$. The pair $\w{\mu}_*$ is said to be a tangent pair to $\w{\mu}$ in the point $a$ if there exists a sequence $r_i\searrow 0$ with $\w{\mu}_{a,r_i}\rightrightarrows\w{\mu}_*$. Observe that \eqref{eq:maloscdotangent} is scaling invariant, therefore \eqref{eq:maloscdotangent} holds as well for the rescaled pairs $\w{\mu}_{a,r}$. Thus, up to a subsequence, the limit always exists.
\end{definition}

The next lemma is an immediate consequence of \cite[Lemma 3.3]{Scheven}.
\begin{lemma}\label{le:convergenceofpairs}
 Let $\w{\mu}_i\in \w{\mathcal{B}}_\Lambda$, $\w{u}\in W^{2,2}(B_4,\n)$ and $\w{\nu}$ be a Radon measure on $\overline{B}_4$. If $\w{\mu}_i\rightrightarrows\w{\mu} = (\w{u}, \w{\nu})$ as $i\rightarrow\infty$, then $\w{\mu}\in\w{\mathcal{B}}_{\Lambda}$. In particular, if $\w{\mu}_* = (\w{u}_*,\w{\nu}_*)$ is a tangent pair of $\w{\mu} = (\w{u}, \w{\nu})\in \w{\mathcal{B}}_\Lambda$ in some point $a\in B_1$, then $\w{\mu}_*\in\w{\mathcal{B}}_\Lambda$.
\end{lemma}

For a pair $\w{\mu} = (\w{u}, \w{\nu})\in \w{\mathcal{B}}_{\Lambda}$ we define the set $\Sigma_{\w{\mu}}$ as the set of points $a\in\overline{B}_1$ with
\begin{equation}\label{eq:singularsetforpairs}
 \liminf_{\rho\searrow 0}\brac{\rho^{4-m}\int_{B\rho(a)}\brac{|\nabla^2 \w{u}|^2 + \rho^{-2}|\nabla \w{u}|^2}dx + \rho^{4-m}\w{\nu}(B_\rho(a))}\ge \frac{\ve_0}{2},
\end{equation}
where the constant $\ve_0$ is the constant introduced in Corollary \ref{co:odpschcor27}. We observe that theorem on the structure of defect measures \cite[Theorem 3.4]{Scheven} carries over directly to our setting to yield
\begin{theorem}\label{thm:structureofdefectmeasures}
 For any $\w{\mu} = (\w{u}, \w{\nu})\in \w{\mathcal{B}}_\Lambda$, there holds $\Sigma_{\w{\mu}} = \mbox{sing}\,{\w{u}}\cup \mbox{spt}\,\w{\nu}$, in particular $\Sigma_{\w{\mu}}$ is a closed set. Moreover, there are constants $c$ and $C$ depending only on $m$ such that for every $\w{\mu} = (\w{u}, \w{\nu})\in\w{\mathcal{B}}_\Lambda$, we have 
 \begin{equation}\label{eq:thmstructuresingset}
  c\, \ve_0\h^{m-4}\mr \Sigma_{\w{\mu}} \le \w{\nu}\mr \overline{B}_1 \le C\Lambda \h^{m-4}\mr \Sigma_{\w{\mu}}.
 \end{equation}
For any sequence $\w{\mathcal{B}}_{\Lambda}\ni(\w{u}_i,0)\rightrightarrows(\w{u}, \w{\nu})$ as $i\rightarrow\infty$, we have subconvergence $\w{u}_i\rightarrow \w{u}$ in $C^2_{loc}(B\setminus\Sigma_{\w{\mu}})$.
\end{theorem}

\begin{proof}
We proceed as in \cite[proof of Theorem 3.4]{Scheven}. The proof of the inclusion $\Sigma_{\w{\mu}}\subset \text{sing}\,\wu\cup \text{spt}\,\w{\nu}$ and the estimates \eqref{eq:thmstructuresingset} remain unchanged, therefore we omit this part. To proof the inclusion $\brac{\text{sing}\,\wu\cup\text{spt}\,\w{\nu}}\subset\Sigma_{\w{\mu}}$ and the subconvergence we follow Scheven with the following modifications and adjustments. 

We divide the proof into three cases: $a\in T_1\setminus\Sigma_{\w{\mu}}$, 
$a\in B^+\setminus \Sigma_{\w{\mu}}$, and $a\in B^-\setminus \Sigma_{\w{\mu}}$.

First, if we choose $a\in T_1\setminus \Sigma_{\w{\mu}}$, then the difference in the proof is the following: We choose a radius $0<\rho<1$ with
\[
 (2\rho)^{4-m}\int_{B_{2\rho}(a)}\brac{|\nabla^2 \wu|^2 + (2\rho)^{-2}|\nabla \wu|^2}\dx + (2\rho)^{4-m}\w{\nu}(B_{2\rho}(a))<\frac{\ve_0}{2}.
\]
Next, we choose a sequence of minimizing biharmonic maps $u_i\in W^{2,2}(B^+_4,\n)$ with boundary data $\vp_i$ with $(\wu_i,0)\rightrightarrows(\wu,\w{\nu})=\w{\mu}$ with
\[
 \lim_{i\rightarrow\infty} (2\rho)^{4-m}\int_{B_{2\rho}(a)}\brac{|\nabla^2 \wu_i|^2+(2\rho)^{-2}|\nabla \wu_i|^2}\dx<\frac{\ve_0}{2}.
\]
Now, in order to apply Corollary \ref{co:odpschcor27} we estimate
\[
 \begin{split}
  \lim_{i\rightarrow\infty} (2\rho)^{4-m}&\int_{B_{2\rho}^+(a)}\brac{|\nabla^2 u_i|^2+(2\rho)^{-2}|\nabla u_i|^2}\dx\\
  &\le \lim_{i\rightarrow\infty} (2\rho)^{4-m}\int_{B_{2\rho}(a)}\brac{|\nabla^2 \wu_i|^2+(2\rho)^{-2}|\nabla \wu_i|^2}\dx + C\rho^2\norm{\vp_i}^2_{C^2}\\
  &< \ve_0.
 \end{split}
\]
Hence, we obtain uniform estimates $\sup_{i\in \N}\norm{u_i}_{C^3(B^+_\sigma(a),\n)}\le C(\n)$ on some smaller half-ball $B^+_\sigma(a)\subset B^+_\rho(a)$. Since the boundary conditions $\vp_i$ are smooth and uniformly bounded we obtain as well $\sup_{i\in\N}\norm{u_i-\vp_i}_{C^3(B_\sigma^+(a),\R^\ell)}<C(\n)$. Now, due to the properties of the reflection \eqref{def:reflection}, we have $\w{u}_i\in C^3(B_\sigma(a))$ with the estimate 
\[
 \sup_{i\in\N}\norm{\wu_i}_{C^3(B_\sigma(a),R^\ell)}<C(\n).
\]
Now, similarly as in \cite{Scheven} by Arzel\`{a}-Ascoli theorem we find a subsequence, which converges $\w{u}_{i_j}\rightarrow \wu$ in $C^2(B_\sigma(a),\R^\ell)$, as $j\rightarrow\infty$, from which we deduce $\w{\nu}(B_\sigma(a))=0$. Thus, $\brac{\text{sing}\,\wu\cup\text{spt}\,\w{\nu}}\subset\Sigma_{\w{\mu}}$.

If we choose $a\in B^+\setminus \Sigma_{\w{\mu}}$ then the proof is identical as in the case of interior points in \cite{Scheven}.

Finally, if we choose $a\in B^-\setminus \Sigma_{\w{\mu}}$ and $\rho$ small enough to ensure $\frac{a_m}{4}+2\sigma <0$, where $a_m$ is the $m$-th component of $a$, then $B_{2\rho}(a)\subset B^-$. By the definition the behavior of the reflected map on $B_{2\rho}(a)\subset B^-$ corresponds to the behavior of the map on four balls in the upper half: $B_{2\rho}(a_j)$, where $a_j = (a',-a_m/j)$ for $j=1,2,3,4$, and $a=(a',a_m)$. Thus from $a\notin \Sigma_{\w{\mu}}$ we deduce that $a_j\notin \Sigma_{\w{\mu}}$ and by repeating the proof in the interior case we obtain the desired inclusion.
\end{proof}

As a consequence of Theorem \ref{thm:structureofdefectmeasures} we obtain, exactly as in \cite[Corollary 3.6]{Scheven}, the following. 
\begin{corollary}\label{co:odpschco36}
If $\w{\mathcal{B}}_\Lambda \ni (\w{u}_i,0)\rightrightarrows(\w{u},\w{\nu})=\w{\mu}$ as $i\rightarrow\infty$, then $\h^{m-4}(\Sigma_{\w{\nu}})=0$ implies $\w{u}_i\rightarrow\w{u}$ strongly in $W^{2,2}(B_{1/2},\n)$. Conversely, the strong convergence $\w{u}_i\rightarrow \w{u}$ in $W^{2,2}(B,\n)$ implies $\h^{m-4}(\Sigma_{\w{\mu}}\cap B)=0$.
\end{corollary}

We also have a counterpart of \cite[Lemma 3.7]{Scheven}, which makes it possible to restrict our attention to the case, when the limiting map is constantly equal 0 and the defect measure is flat.
\begin{lemma}\label{le:odpschle37}
 Assume there is a pair $(\w{u},\w{\nu})\in\w{\mathcal{B}}_\Lambda$ with $\h^{m-4}(\text{spt}\,\w{\nu})>0$. Then there is a pair $(\w{u}_*,\overline{\nu})\in\w{\mathcal{B}}_\Lambda$, such that $\w{u}_*=0\in\R^\ell$ and 
 \[
  \overline{\nu} = C \h^{m-4}\mr V,
 \]
where $V$ is an $(m-4)$-dimensional subspace $V\subset \R^m$ and $C>0$ is a constant.
\end{lemma}
 \begin{proof}
 The proof follows directly the proof of Scheven's Lemma 3.7 \cite{Scheven}. Identically as there, there is a point $a\in B$ and a sequence $r_i\searrow0$ for which $\w{\mu}_{a,r_i}\rightrightarrows\w{\mu}_* = (\w{u}_*,\w{\nu}_*)\in\w{\mathcal{B}}_\Lambda$, for which $\w{u}_*$ is a constant. We know also that $\w{u}_*$ is equal zero on $T_4$, thus $\w{u}_* =0$. 
 
In the proof of the structure of the measure $\overline{\nu}$ the only difference from Scheven's proof we should observe is that, by inequality \eqref{eq:measureinequality}, the quantity
 \[
  \w{\Theta}^{m-4}(\w{\nu}_\ast,a):=\lim_{\rho\searrow0}\rho^{4-m}\w{\nu}_\ast(B_\rho(a))
 \]
is well defined and a similar analysis to that in \cite{Scheven} shows that there exists a tangent measure $\overline{\nu}$ to $\w{\nu}_*$, such that $\overline{\nu} = C\h^{m-4}\mr V$, where $V$ is an $(m-4)$-dimensional subspace $V\subset \R^m$.
\end{proof}

We are ready to prove the Compactness Theorem \ref{thm:comp}. The (rough) idea of the proof follows Scheven's proof of  Theorem 1.5 in \cite{Scheven}. The results of this section yield that if the theorem was false we would obtain a sequence of reflections $\wu_i$, converging to 0 off the support of a defect measure which up to a constant is an $(m-4)$-dimensional Hausdorff measure and which is flat. To show that it is impossible we construct a comparison map and use the minimizing property of $u_i$ on a half-ball. We define a comparison map as an interpolation between $u_i$ and its boundary data $v_i$ with the exception of a tori of small radius in which the energy concentration set is included. To define the map on the remaining tori we use a kind of radially constant extension on a tori. The existence of such a map leads to a contradiction with the special form of the defect measure, if we choose sufficiently small outer annuli on which the comparison map is equal $u_i$ and sufficiently small intermediate annuli on which the map is defined as an interpolation between $u_i$ and $\vp_i$. 

\begin{proof}[Proof of Theorem \ref{thm:comp}]
In the following, \emph{we forego possibly more general and sophisticated estimates in favor of simple arithmetic}\footnote{see Iwaniec \cite[Note, p.607]{Iwaniec}}. 

First lines of our proof are essentially the same as the first 19 lines of Scheven's proof: we argue by contradiction and collect the results of this section.

The theorem is equivalent to $\w{\nu}\mr\overline{B}\equiv 0$ for all $(\w{u}, \w{\nu})\in\w{\mathcal{B}}_\Lambda$. Thus, we consider a sequence ${w}_i\in W^{2,2}(B_4^+,\n)$ of minimizing biharmonic maps with some boundary values such that $\sup_i\norm{\w{w}_i}_{W^{2,2}(B_4^+)}<\infty$. 
By Lemma \ref{le:boundednessinmorrey} we know that  $\sup_i\norm{\w{w}_i}_{L^{2,m-4}}<\Lambda$ for some $\Lambda>0$, so that after passing to a subsequence  we have  $(\w{w}_i,0)\rightrightarrows(\w{w},\w{\nu})=\w{\mu}\in\w{\mathcal{B}}_\Lambda$. 
Assume, on the contrary, that we do not have strong convergence $\w{w}_i\rightarrow \w{w}$ in $W^{2,2}(B_{1/2},\n)$. Then, by Corollary \ref{co:odpschco36} we know that $\h^{m-4}(\Sigma_{\w{\mu}})>0$. 
By Lemma \ref{le:odpschle37} we know that there are minimizing biharmonic maps $u_i\in W^{2,2}(B_4^+,\n)$
with boundary values $\vp_i$, such that $(\w{u}_i,0)\rightrightarrows(0,\overline{\nu})$ and $\overline{\nu}\mr \overline{B} = C\h^{m-4}\mr V$. Moreover, by Theorem \ref{thm:structureofdefectmeasures} we get
\begin{equation}\label{eq:pfofthmstronconvinC2}
 \w{u}_i\rightarrow 0 \quad \text{ in } C^2_{loc}\brac{B\setminus V} \text{ as } i\rightarrow\infty.
\end{equation}
Now let us note that if the energy concentration set $V$ would be a subset of $T_1$ we would be in the simplest situation as our map $\w{u}_i$ vanish on $T_1$. Therefore, without loss of generality we may assume that 
\[
 V = \{0\}\times \overline{B}^{m-4}.
\]
Let $\kappa$, $\sigma$ be suitable parameters, which will be specified later, satisfying \[\frac12<\kappa<1, \quad 0<\sigma<\frac{1}{16}, \quad \text{ and } 0<\kappa + 2\sigma < 1.\] Let $\psi\in C^\infty(B^+,[0,1])$ be a cut-off function with $\psi \equiv 0$ on $B^+_{\kappa-\sigma}$, $\psi\equiv 1$ outside $B^+_{\kappa+\sigma}$ and $|\nabla \psi|\le \frac{C}{\sigma}$, $|\nabla^2 \psi|^2\le \frac{C}{\sigma^2}$ on $B^+$.  We employ Scheven's notation of tori: for $(X_1,X_2)\in \R^4\times\R^{m-4}$, we define \[\mathbb{T}_{s}:=\{x\in\R^m:[x]\le s\},\quad \text{ where } [x]:=[|X_1|^2+(|X_2|-\kappa)^2]^{1/2}.\] 

Now let 
\begin{equation}\label{eq:vtilda}
 \w{v}_i(x) := \pi_\n (\vp_i + \psi (\wu_i)) \quad \text{ for } x\in B^+\setminus \mathbb{T}_{2\sigma}.
\end{equation}
We recall that $\pi_\n:O(\n)\rightarrow\n$ is the nearest point projection of a neighborhood $O(\n)\subset \R^\ell$ of $\n$ onto $\n$. We observe that
\[
 \w{v}_i\equiv \vp_i \text{ on } B^+_{\kappa-\sigma} \quad \text{ and } \quad \w{v}_i\equiv u_i \text{ outside } B^+_{\kappa+\sigma}.  
\]
Moreover, the set $\{0<\psi<1\}\setminus \mathbb{T}_{2\sigma}$ has positive distance to the energy concentration set $\{0\}\times B^{m-4}$, so that we have convergence $\w{u}_i\rightarrow 0$ in $C^2$ on the former set. Therefore, for sufficiently large $i\in\N$, the maps $\w{v}_i(x)$ are well defined for $x\in B^+\setminus \mathbb{T}_{2\sigma}$.

Simple computations yield 
\begin{equation}\label{eq:vtildawithouttsigma}
 \begin{split}
  &\int_{B^+\setminus \mathbb{T}_{2\sigma}} |\Delta \w{v}_i|^2 \dx \le \int_{B^+\setminus B^+_{\kappa}}|\Delta u_i|^2 \dx  + \int_{B^+_{\kappa}}|\Delta \vp_i|^2 \dx \\
&\quad + C\int_{\{0<\psi<1\}\setminus \mathbb{T}_{2\sigma}}\brac{|\Delta \vp_i|^2 + |\nabla \vp_i|^4}\dx\\
&\quad + C\int_{\{0<\psi<1\}\setminus \mathbb{T}_{2\sigma}}\brac{|\Delta\wu_i|^2 + |\nabla \vp_i|^2|\nabla \wu_i|^2 + \frac{|\nabla \wu_i|^2}{\sigma^2} + \frac{|\wu_i|^4 }{\sigma^4} +\frac{|\wu_i|^2}{\sigma^4} }dx\\
 &\le m \int_{B\setminus B_{\kappa}}|\nabla^2 \wu_i|^2 \dx  + \int_{B^+}|\Delta \vp_i|^2\dx +C\int_{\{0<\psi<1\}}\brac{|\nabla \vp_i|^4 + |\nabla \vp_i|^6}\dx \\
&\quad + C\int\limits_{\{0<\psi<1\}\setminus \mathbb{T}_{2\sigma}}\brac{ |\Delta\wu_i|^2 + |\nabla \wu_i|^3 + \frac{|\nabla \wu_i|^2}{\sigma^2} + \frac{|\wu_i|^4}{\sigma^4} + \frac{|\wu_i|^2}{\sigma^4} }\dx,
 \end{split}
\end{equation}
where in the last estimate we used Young's inequality \[|\nabla \vp_i|^2|\nabla \wu_i|^2\le \frac{|\nabla \vp_i|^6}{3} + \frac{2|\nabla \wu_i|^3}{3}\]
and the pointwise inequality $|\Delta u_i|^2\le m|\nabla^2 u_i|^2$. 

By Gagliardo-Nirenberg's inequality we have for $p>1$
\begin{equation}\label{eq:gnineqp}
 \norm{\w{u}_i}_{L^{2p}(B^+)}^2  \le C \norm{\wu_i}_{L^\infty(B^+)}\norm{\wu_i}_{W^{2,p}(B^+)}.
\end{equation}
For $1<p<2$ the uniform boundedness of $\norm{\wu_i}_{W^{2,p}(B^+)}$ combined with the above inequality implies $\sup_{i\in\N}\norm{\wu_i}_{L^{2p}(B^+)}<\infty$. Recall that $(\wu_i,0)\rightrightarrows(0,\overline{\nu})$, thus $\wu_i\rightarrow 0$ strongly in $W^{1,p}(B^+)$. Now, by H\"{o}lder's inequality for exponent $\frac72$
 \begin{equation}\label{eq:convuto3}
 \begin{split}
 \int_{B^+} |\nabla \wu_i|^3 \dx &= \int_{B^+} |\nabla \wu_i|^{1/2}|\nabla \wu_i|^{5/2}\dx\\
 &\le \brac{\int_{B^+} |\nabla \wu_i|^{7/4}\dx}^{2/7}\brac{\int_{B^+} |\nabla \wu_i|^{7/2}\dx}^{5/7}.
 \end{split}
 \end{equation}
Taking $p=\frac74$ we see that the first term of the latter inequality converges to 0 and the second is bounded, hence $|\nabla \wu_i|^3\rightarrow 0$ strongly in $L^3(B^+)$. Now we are ready to pass to the limit in \eqref{eq:vtildawithouttsigma}.
 
By the $C^2$-convergence $\w{u}_i\rightarrow 0$ on the set $\{0<\psi<1\}\setminus \mathbb{T}_{2\sigma}$ and by the convergence $|\Delta \w{u}_i|^2\dx\rightharpoonup \overline{\nu}$ in the sense of measures, we get, since $\overline{\nu}(\partial B) = 0$,
\begin{equation}\label{eq:limitvtilda}
 \lim_{i\rightarrow\infty}\int_{B^+\setminus \mathbb{T}_{2\sigma}}|\Delta v_i|^2\dx \le m\overline{\nu}(B\setminus B_{\kappa}) + C_\vp +C(\sigma),
\end{equation}
where $C(\sigma)$ is the limit of $C\int_{\{0<\psi<1\}}\brac{|\nabla \vp_i|^4 + |\nabla \vp_i|^6}\dx$. From the absolute continuity of the Lebesgue integral, by shrinking $\sigma>0$, the constant $C(\sigma)$ can be taken arbitrary small.

Note that for $m=4$ the above construction of $\w{v}_i$ is possible for all $x\in B^+$. In this situation $\w{v}_i \equiv u_i \text{ on } T_1$ and since the vector $\frac{\partial u_i}{\partial x_m}$ is perpendicular to the manifold $\n$ we have
\[
 \frac{\partial}{\partial x_m} \w{v}_i\bigg\rvert_{T_1} \ =\frac{\partial }{\partial x_m }u_i\bigg\rvert_{T_1},
 \]
see, e.g., \cite[Section 2.12.3]{Simon1996}. Thus, $u_i -\w{v}_i\in W^{2,2}_0(B^+)$ and we can use the minimality of $u_i$ to compare it with $\w{v}_i$. 

Applying Lemma \ref{a:secondderivativesbylaplacian} for $R=1$ and $r=1-\sigma$, using the inequality  \eqref{eq:reflectionnormestimate}, and the strong convergence of $\nabla \wu_i$ in $L^2$, we get
\begin{equation}\label{eq:zwartoscporownaniem=4}
\begin{split}
 \overline{\nu}(B_{1-\sigma}) &= \lim_{i\rightarrow\infty} \int_{B_{1-\sigma}}|\nabla^2 \wu_i|^2 \dx \le \lim_{i\rightarrow\infty} 2\int_{B} \brac{|\Delta \w{u}_i|^2 + \frac{C}{\sigma^2}|\nabla \wu_i|^2} \dx \\
 &\le \lim_{i\rightarrow\infty} 2 C_{ref} \brac{\int_{B^+} |\Delta \w{u}_i|^2\dx + \int_{B}\frac{C}{\sigma^2}|\nabla \wu_i|^2\dx}\\
 &\le \lim_{i\rightarrow\infty} 4C_{ref}\brac{\int_{B^+}|\Delta u_i|^2\dx + \int_{B^+} |\Delta \vp_i|^2 \dx+\int_{B}\frac{C}{\sigma^2}|\nabla \wu_i|^2\dx}\\
& \le \lim_{i\rightarrow\infty} 4C_{ref}\brac{\int_{B^+} |\Delta \w{v}_i|^2 \dx +  \int_{B^+} |\Delta \vp_i|^2 \dx+\int_{B}\frac{C}{\sigma^2}|\nabla \wu_i|^2\dx}\\\
&\le 4m\,C_{ref} \overline{\nu}(B\setminus B_{\kappa}) + 4C_{ref}C_\vp +C(\sigma)\\
&< \overline{\nu}(B_{1-\sigma}), 
\end{split}
\end{equation}
a contradiction with the form of $\w{\nu}$. For the last inequality we needed $C_\vp<\frac{1}{4C_{ref}}$ and a sufficiently small $\sigma$.

For $m\ge 5$ we apply a retraction $\Psi\in C^\infty(\mathbb{T}_{4\sigma}\setminus \mathbb{T}_0, \mathbb{T}_{4\sigma}\setminus \mathbb{T}_{2\sigma})$ from \cite[Lemma 3.8]{Scheven} with the following properties:
$\Psi = id$ and $\nabla \Psi\equiv Id$ on $\partial \mathbb{T}_{4\sigma}$, where $id$ denotes the identity on $\partial \mathbb{T}_{4\sigma}$ and $Id$ is the identity map on $\R^m$. Furthermore, 
\begin{equation}\label{eq:propertiesofPsi}
 |\nabla\Psi(x)|\le\frac{C\sigma}{[x]},\quad |\nabla^2 \Psi(x)|\le \frac{C\sigma}{[x]^2}, \text{ and } \det (\nabla\Psi(x))\ge \frac{C\sigma^4}{[x]^4}
\end{equation}
for constants dependent only on $m$. 

We are ready to define a comparison map. Let
\begin{equation}
 v_i(x) : = \left\{ \begin{array}{ll}
     \pi_\n (\vp_i(x) + \psi(x) \wu_i(x)) & \textrm{ for }x\in B^+\setminus \mathbb{T}_{4\sigma},\\
     \pi_\n (\vp_i(x) + \psi(x) \wu_i(\Psi(x))) & \textrm{ for }x\in \mathbb{T}_{4\sigma},
  \end{array} \right.
\end{equation}
i.e., $v_i(x) = \w{v}_i(x)$ on $B^+\setminus \mathbb{T}_{4\sigma}$. Due to the properties of the retraction $\Psi$ we have $v\in W^{2,2}(B^+,\n)$. We immediately have
\[
 \brac{v_i,\frac{\partial }{\partial x_m}v_i}\Bigg\rvert_{T_1\setminus \mathbb{T}_{4\sigma}} = \brac{u_i,\frac{\partial }{\partial x_m}u_i}\Bigg\rvert_{T_1\setminus \mathbb{T}_{4\sigma}}.
\]
To see that the trace of $v_i$ is the same as $u_i$'s on $T_1\cap \mathbb{T}_{4\sigma}$ we note that for $\Psi$ from Scheven's Lemma 3.8 we have 
\[
 x\in T_1\cap \mathbb{T}_{4\sigma} \Rightarrow \Psi(x)\in T_1.
\]
Thus, \[u_i(\Psi(x)) - \vp_i(\Psi(x)) =0, \quad \nabla u_i(\Psi (x)) = \nabla (\Psi (x)) = 0 \quad \text{ for } x\in T_1\cap \mathbb{T}_{4\sigma}.\]
Hence, after simple computations, for  $ x\in T_1\cap \mathbb{T}_{4\sigma}$,
\[
 v_i(x) = u_i(x) , \quad \frac{\partial}{\partial x_m}v_i(x) = (\pi_\n)_{x_k}\frac{\partial u_i^k(x)}{\partial x_m} = \frac{\partial u_i(x)}{\partial x_m}.
\]
The last equality is, again, a consequence of the fact that $\frac{\partial u_i}{\partial x_m}\perp\n$. 

Similarly as in \eqref{eq:vtildawithouttsigma} we compute 
\begin{equation}
 \begin{split}
  &\int_{\mathbb{T}_{4\sigma}^+}|\Delta v_i|^2\dx \\
  &\le C\int_{\mathbb{T}_{4\sigma}^+}\brac{|\Delta \vp_i|^2 +|\nabla \vp_i|^4 + |\nabla \vp_i|^6}\dx \\
  &\quad + C\int_{\mathbb{T}_{4\sigma}^+} \bigg(|\nabla^2 \wu_i\circ\Psi|^2|\nabla\Psi|^4 + |\nabla \wu_i\circ\Psi|^3|\nabla \Psi|^3 \\
  &\phantom{\quad + C\int_{T_{4\sigma}^+}}\quad + |\nabla \wu_i\circ\Psi|^2\brac{\frac{|\nabla \Psi|^2}{\sigma^2} + |\nabla^2 \Psi|^4} + \frac{|\wu_i\circ \Psi|^4}{\sigma^4} + \frac{|\wu_i\circ \Psi|^2}{\sigma^4}\bigg) \dx.
 \end{split}
\end{equation}
Using the properties \eqref{eq:propertiesofPsi} of $\Psi$ and the fact that $[x]\le \frac{1}{4}$ we get
\begin{equation}
 \begin{split}
  &\int_{\mathbb{T}_{4\sigma}^+}|\Delta v_i|^2\dx - C\int_{\mathbb{T}_{4\sigma}^+}\brac{|\Delta \vp_i|^2 +|\nabla \vp_i|^4+ |\nabla \vp_i|^6}\dx\\
&\le C\int_{\mathbb{T}_{4\sigma}^+} \bigg(\frac{\sigma^4}{[x]^4}|\nabla^2 \wu_i\circ\Psi|^2 +\frac{\sigma^3}{[x]^4}|\nabla \wu_i\circ\Psi|^3 +\frac{1}{[x]^4}|\nabla \wu_i\circ\Psi|^2\\
&\phantom{\quad + C\int_{T_{4\sigma}^+}}\quad+ \frac{\sigma^2}{[x]^4}|\nabla \wu_i\circ\Psi|^2 +\frac{\sigma^{-4}}{[x]^4}|\wu_i\circ \Psi|^4 +\frac{\sigma^{-4}}{[x]^4}|\wu_i\circ \Psi|^2 \bigg)\dx \\
&\le C\int_{\mathbb{T}_{4\sigma}^+} \bigg(|\nabla^2 \wu_i\circ\Psi|^2 +\sigma^{-1}|\nabla \wu_i\circ\Psi|^3 +\sigma^{-4}|\nabla \wu_i\circ\Psi|^2\\
&\phantom{\quad + C\int_{T_{4\sigma}^+}}\quad+ \sigma^{-2}|\nabla \wu_i\circ\Psi|^2 +\sigma^{-8}|\wu_i\circ \Psi|^4 +\sigma^{-8}|\wu_i\circ \Psi|^2 \bigg)\det(\nabla \Psi)\dx\\
&\le C\int_{\mathbb{T}_{4\sigma}^+\setminus \mathbb{T}_{2\sigma}^+} \bigg(|\nabla^2 \wu_i|^2 +\sigma^{-1}|\nabla \wu_i|^3 +\sigma^{-4}|\nabla \wu_i|^2+ \sigma^{-8}|\wu_i|^4 +\sigma^{-8}|\wu_i|^2 \bigg)\dx.
 \end{split}
\end{equation}
In order to pass with $i$ to the limit in the above inequality we note that similarly as in \eqref{eq:limitvtilda}, we have $\int_{\mathbb{T}_{4\sigma}^+\setminus \mathbb{T}_{2\sigma}^+}|\nabla^2 \wu_i|^2\dx \le \int_{\mathbb{T}_{4\sigma}\setminus \mathbb{T}_{2\sigma}}|\nabla^2 \wu_i|^2\dx$. Thus, 
\begin{equation}\label{eq:limitviT}
\begin{split}
 \lim_{i\rightarrow\infty}\int_{\mathbb{T}_{4\sigma}^+} |\Delta v_i|^2\dx &\le C\int_{\mathbb{T}^+_{4\sigma}}(|\Delta \vp|^2 + |\nabla \vp|^4+|\nabla \vp|^6)\dx+ C\, \overline{\nu}(\mathbb{T}_{4\sigma})\\
 &= C(\sigma) + C\, \overline{\nu}(\mathbb{T}_{4\sigma}).
 \end{split}
\end{equation}
Once again, from the absolute continuity of the Lebesgue integral, by shrinking $\sigma>0$,  the constant $C(\sigma)$ can be taken arbitrary small.

Now let $0<\gamma<1$ be a small number. We have by Lemma \ref{a:secondderivativesbylaplacian}
\begin{equation}\label{eq:seconddervbylap}
 \int_{B^+_{1-\gamma}}|\nabla \wu_i|^2\dx \le 2\int_{B^+}\brac{|\Delta \wu_i|^2 + \frac{C}{\gamma^2}|\nabla \wu_i|^2}\dx.
\end{equation}
Combining \eqref{eq:reflectionnormestimate} and \eqref{eq:seconddervbylap} we obtain 
\begin{equation}\label{eq:dowodzwartosciporownanie}
\begin{split}
 \overline{\nu}(B_{1-\gamma}) &=  \lim_{i\rightarrow\infty} \int_{B_{1-\gamma}}|\nabla^2 \wu_i|^2 \dx \le \lim_{i\rightarrow\infty} C_{ref}\int_{B^+_{1-\gamma}}|\nabla^2 \wu_i|^2 \dx\\
 &\le \lim_{i\rightarrow\infty}2\,C_{ref}\int_{B^+}\brac{|\Delta \wu_i|^2 + \frac{C}{\gamma^2}|\nabla \wu_i|^2}\dx \\
 &\le \lim_{i\rightarrow\infty} 4\,C_{ref}\brac{\int_{B^+}|\Delta u_i|^2\dx + \int_{B^+} |\Delta \vp_i|^2 \dx +\int_{B^+} \frac{C}{\gamma^2}|\nabla \wu_i|^2\dx}\\
& \le \lim_{i\rightarrow\infty} 4\,C_{ref}\brac{\int_{B^+} |\Delta v_i|^2 \dx + \int_{B^+} |\Delta \vp_i|^2 \dx +\int_{B^+} \frac{C}{\gamma^2}|\nabla \wu_i|^2 \dx}\\
&\le 4m\,C_{ref} \overline{\nu}(B\setminus B_{\kappa}) + 4m\,C_{ref}(C_\vp +C(\sigma)+ C\,\overline{\nu}(\mathbb{T}_{4\sigma})).
\end{split}
\end{equation}
To get a contradiction we use the special form of the measure \[\overline{\nu}\mr\overline{B} = C\h^{m-4}\mr\brac{\{0\}\times\overline{B}^{m-4}}.\] 
We choose the number $\kappa$ so that $4m\, C_{ref}\overline{\nu}(B\setminus B_\kappa) < \overline{\nu}(B_{1-\gamma})$, for example
\[
\kappa = \sqrt[m-4]{1-\frac{(1-\gamma)^{m-4}}{3500m}}.
\]
Next we observe that if $C_\vp$ is sufficiently small, e.g, is such that
\[
 4m\,C_{ref} C_\vp <\frac12 \brac{\overline{\nu}(B_{1-\gamma})-4m\, C_{ref}\overline{\nu}(B\setminus B_\kappa )}
\]

then by shrinking $\sigma>0$ the number $4m\,C_{ref}(C(\sigma) + C\,\overline{\nu}(\mathbb{T}_{4\sigma}))$ can be arbitrary small and thus
\[
4m\,C_{ref} \overline{\nu}(B\setminus B_{\kappa}) + 4m\,C_{ref}(C_\vp +C(\sigma)+ C\,\overline{\nu}(\mathbb{T}_{4\sigma})) < \overline{\nu}(B),
\]
contradicting \eqref{eq:dowodzwartosciporownanie}. This finishes the proof of Theorem \ref{thm:comp}.

\end{proof}

\section{Tangent maps at the boundary}\label{s:tangentmaps}
In this section we prove, using the compactness result from the previous section, that limits of rescaled maps converge strongly to boundary tangent maps, which are homogeneous of degree 0 and have constant values on the flat part of the boundary $\partial B^+$. Next, we show how to rule out the possibility of existence of nonconstant minimizing biharmonic maps from a half ball, that are constant on the flat part of the boundary $T_1$. Finally, combining results of this section, Scheven's lemma, which states that the tangent maps that occur in the dimension reduction argument are minimal, and Gong, Lamm, and C.~Wang's epsilon regularity result --- Lemma \ref{le:eps-reg} --- we give the proof of the main result.

\begin{definition}
 Let $a\in T_1$ and $x\in \frac{1}{\lambda}(B^+_1-a)$. We define the rescaled map by
\[
 u_{a,\lambda}(x) := u(a+\lambda x).
\]
A map $v\in W^{2,2}_{loc}(\R^m, \n)$ is called a \emph{tangent map at the boundary} of $u$ at the point $a$ if there exists a sequence $\lambda_i\searrow0$ with $u_{a,\lambda_i}\rightarrow v$ in $W^{2,2}_{loc}(\R^m, \n)$ as $i\rightarrow\infty$.  
\end{definition}

\begin{lemma}\label{le:homocompstrong}
 Let $u$ be as before with boundary values $\varphi\in C^\infty$ and let $a\in T_1$. Then for each sequence $\{\lambda_i\}_{i\in\N}$ for which $0<\lambda_i<1$, there exists a subsequence $\lambda_{i_j}\rightarrow0$ such that the maps $u_{\lambda_{i_j}}$ converge strongly in $W^{2,2}(B^+_{1/2},\n)$ to a map $u_0\in W^{2,2}(B^{+},\n)$ that is biharmonic, homogeneous of degree 0, and has constant boundary values on $T_1$. 
\end{lemma}

\begin{proof}
\textsc{Step 1: Strong convergence.} Observe that $\sup_i\norm{u_{a,\lambda_i}}_{W^{2,2}(B^+)}<\infty$. Indeed, by a change of variables 
\[
  \norm{\nabla^2 u_{a,\lambda_i} }^2_{L^2(B_1^+)} =\int_{B_1^+}|\lambda_i^2 \nabla^2 u(a+\lambda_i x) |^2 \dx = \lambda_i^{4-m}\int_{B^+_{\lambda_i}(a)}|\nabla^2 u|^2\,dy.
\]
The supremum of the latter one is bounded by Lemma \ref{le:boundednessinmorrey}. Moreover, $u_0\big\rvert_{T_1}=\varphi(a)$ by the continuity of $\varphi$. Thus the assumptions of Theorem \ref{thm:comp} are satisfied and we obtain the strong subconvergence to $u_0$.

\textsc{Step 2:  Homogeneity of degree 0.}
By strong convergence and Lemma \ref{le:convergenceofmonocomp} we have
\begin{align*}
H^+_{u_0}(0,r) &= \lim_{i\rightarrow\infty}H^+_{u_{a,r_i}}(0,r) = \lim_{i\rightarrow\infty}H^+_{u}(a,rr_i) = \lim_{\rho\searrow0}H^+_u(a,\rho)\\
R^+_{u_0}(0,r)&= \lim_{i\rightarrow\infty}R^+_{u_{a,r_i}}(0,r) = \lim_{i\rightarrow\infty}R^+_{u}(a,rr_i) = \lim_{\rho\searrow0}H^+_u(a,\rho).
\end{align*}
Thus, $H^+_{u_0}(0,r),\, R^+_{u_0}(0,r)$ do not depend on $r$ and we denote
\begin{equation}\label{eq:homog1}
 H^+_{u_0}(0) =\lim_{\rho\searrow0}H^+_u(a,\rho), \quad R^+_{u_0}(0)=\lim_{\rho\searrow0}H^+_u(a,\rho).
\end{equation}

Now by monotonicity formula \eqref{eq:boundarymono}
 \begin{equation*}
  \begin{split}
 P^+_{u,\vp}(a,\rho,r) &\le e^{Cr}H^+_{u,\vp}(a,r)+ e^{Cr}R^+_{u,\vp}(a,r) +Cre^{Cr} \\
 & \quad -   H^+_{u,\vp}(a,\rho) - e^{C\rho}R^+_{u,\vp}(a,\rho).
  \end{split}
 \end{equation*}
In particular 
 \begin{equation*}
  \begin{split}
 \int\limits_{B^+_r(a)\setminus B^+_\rho(a)} \frac{\left|(x-a)^i(u-\vp)_{x_i}\right|^2}{|x-a|^m}\dx &\le e^{Cr}H^+_{u,\vp}(a,r)+ e^{Cr}R^+_{u,\vp}(a,r) +Cre^{Cr} \\
 & \quad -   H^+_{u,\vp}(a,\rho) - e^{C\rho}R^+_{u,\vp}(a,\rho).
  \end{split}
 \end{equation*}
By \eqref{eq:homog1} passing to the limit in the last inequality with $\rho\searrow 0$ we obtain
 \begin{equation}\label{eq:homogest}
  \begin{split}
 \int_{B^+_r(a)} \frac{\left|(x-a)^i(u-\vp)_{x_i}\right|^2}{|x-a|^m}\dx &\le e^{Cr}H^+_{u}(a,r)+ e^{Cr}R^+_{u}(a,r) \\
 & \quad -   H^+_{u_0}(0) - R^+_{u_0}(0) +Cre^{Cr}.
  \end{split}
 \end{equation}
 By a change of variables
\begin{equation*}
 \int_{B^+_r(a)} \frac{\left|(x-a)^i(u-\vp)_{x_i}\right|^2}{|x-a|^m}\dx = 
 \int_{B^+_1} \frac{|x^i\brac{u_{a,r}}_{x_i}|^2}{|x|^m}\dx.
\end{equation*}
Thus, passing with $r$ to zero in \eqref{eq:homogest} we get
 \begin{equation*}
\lim_{r\rightarrow 0}\int_{B^+_1} \frac{|x^i\brac{u_{a,r}}_{x_i}|^2}{|x|^m}\dx =0
 \end{equation*}
and as a consequence $x^i (u_0)_{x_i} \equiv 0$ a.e., which implies the desired homogeneity.

\end{proof}

 \begin{lemma}\label{jednstzero}
  Any minimizing biharmonic map $u_0\in W^{2,2}(B^{+}_1,\n)$ that is homogeneous of degree 0 and that is constant on $B_1\cap\{x_m=0\}$ must be a constant. 
 \end{lemma}
\begin{proof}
For $m<4$ by Sobolev embedding theorem a mapping in $W^{2,2}$ must be continuous. Being homogeneous of degree 0, it must be a constant.
 
 For $m=4$ assume, contrary to our claim, that $u_0$ is a nonconstant minimizing map from $B^{+}$ to $\n$. Let $y=\beta(x)=2x$. Simple calculation gives
 \begin{align*}
  0 &< \int_{B^+_1}|\Delta u_0|^2 \, dx=\int_{B^+_{\frac12}}|\Delta u_0(2x)|^2\cdot 2^4 \, dx\\
  & = \int_{B^+_{\frac12}}|\Delta (u_0\circ\beta)(x)|^2|\nabla^2\beta(x)|^{-2}\cdot 2^4 \, dx = \int_{B^+_{\frac12}}|\Delta u_0(x)|^2 \, dx\\
  & <\int_{B^+_1}|\Delta u_0(x)|^2 \, dx <\infty,
 \end{align*}
which is impossible.
 
\begin{figure}[!ht]
\centering
  \includegraphics[width=6cm]{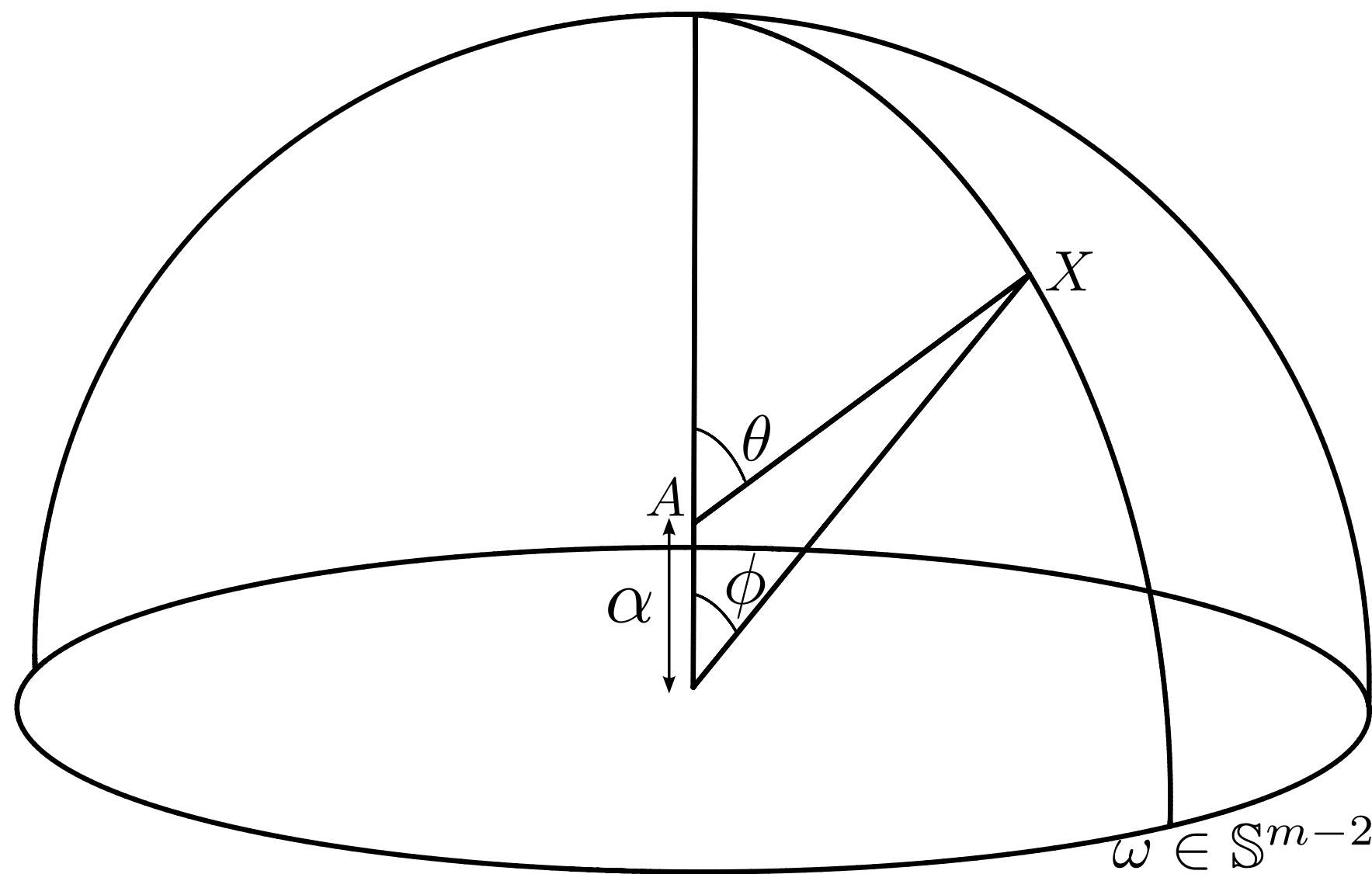}
  \caption{Proof of Lemma 3.3, the relation between $\alpha, \theta$ and $\phi$} \label{gr:tangent}
\end{figure} 
 
For $m>4$, we shall consider the energy of a comparison function $v_\alpha$, the same as in \cite[proof of Theorem 5.7]{HLp}. We use spherical coordinates to represent a point $X$ on the hemisphere $\partial B_1\cap\{x_m\ge 0\}$ by a point $\omega\in\S^{m-2}$ and the angle $\phi\in[0,\frac12 \pi]$. Let  $0<\alpha<1$, $A=(0,\ldots,0,\alpha)$ and $\theta$ denote the angle between vectors $AX$ and $AN$ (where $N=(0,\ldots,0,1)$ is the north pole). The angle $\theta$ satisfies the relation
\begin{equation}
 \theta = \phi + \arcsin(\alpha \sin\theta).
\end{equation}
As the angle $\phi$ ranges between 0 and $\frac12 \pi$, the angle $\theta$ ranges between 0 and $\Theta(\alpha) = \text{arcctg}(-\alpha) = \pi - \arcsin((1+\alpha^2)^{-\frac12})$. The distance between $x$ and $(0,\ldots,\alpha)$ is $R(\phi,\alpha) = [(\alpha - \cos\phi)^2+\sin^2\phi]^{\frac12}$. The desired comparison mapping is given by
\begin{equation}
 v_\alpha(\theta,\omega) = u_0(\phi,\omega) \quad \text{for } \theta\in[0,\Theta] \text{ and } \omega\in\S^{m-2}.
\end{equation}
Let $J(\alpha)= \int_{B^+}|\Delta v_\alpha|^2\dx$ denote the Hessian energy of $v_\alpha$. One can compute
\begin{align*}
 &J(\alpha)=\\ &=\int_0^{\Theta(\alpha)}\int_0^{R(\phi, \alpha)}\int_{\S^{m-2}}\frac{1}{r^4\sin^2\theta} \sum_{i=1}^k\left|\cos\theta \frac{\partial v_\alpha^i}{\partial \theta} + \sin \theta \frac{\partial^2 v^i_\alpha}{\partial\theta^2} + \sin^{-1}\theta \frac{\partial^2 v_\alpha^i}{\partial \omega^2}\right|^2\\
 &\phantom{= \int_0^{\Theta(\alpha)}\int_0^{R(\phi, \alpha)}\int_{\S^{m-2}}}\cdot \sin^{m-2}\theta r^{m-1} \, d\omega\, dr\, d\theta\\
 &=\int_0^{\Theta(\alpha)}\int_{\S^{m-2}} \sum_{i=1}^k\left|\cos\theta \frac{\partial v_\alpha^i}{\partial \theta}+\sin \theta\frac{\partial^2 v_\alpha^i}{\partial \theta^2}+\sin^{-1}\theta\frac{\partial^2 v_\alpha^i}{\partial \omega^2}\right|^2\\
 &\phantom{=\int_0^{\Theta(\alpha)}\int_{\S^{m-2}}}\cdot\sin^{m-4}\theta \frac{1}{m-4} R^{m-4}(\phi,\alpha) \, d\omega \,d\theta.
\end{align*}
Changing variables according to $\theta =\theta (\phi, \alpha):[0,\frac12 \pi]\times[0,1)\rightarrow [0,\Theta(\alpha))$, we find that $J(\alpha)$ equals
\begin{align*}
 &J^i(\alpha) =\\
 &\frac{1}{m-4} \int_0^{\frac{\pi}{2}}\int_{\S^{m-2}}\Bigg|\cos\theta(\phi,\alpha)\frac{\partial u_0^i}{\partial \phi}\frac{\partial \phi}{\partial \theta}\\
 &\phantom{\frac{1}{m-4} \int_0^{\frac{\pi}{2}}\int_{\S^{m-2}}\biggl\lvert}\ +\sin\theta(\phi,\alpha)\left[\frac{\partial^2 u_0^i}{\partial \phi^2}\left(\frac{\partial \phi}{\partial \theta}\right)^2 + \frac{\partial u_0^i}{\partial \phi}\frac{\partial^2 \phi}{\partial \theta^2}\right] + \sin^{-1}\theta \frac{\partial^2 u_0^i}{\partial \omega^2}\Bigg|^2\\
 &\phantom{\frac{1}{m-4} \int_0^{\frac{\pi}{2}}\int_{\S^{m-2}}}\cdot\sin^{m-4}\theta(\phi,\alpha)R^{m-4}(\phi,\alpha)\left|\frac{\partial \theta}{\partial \phi}\right| \, d\omega \, d\phi.
\end{align*}
We denote
\[
K(\alpha,\omega,\phi)= \cos\theta\frac{\partial u_0^i}{\partial \phi}\frac{\partial \phi}{\partial \theta} +\sin\theta\left[\frac{\partial^2 u_0^i}{\partial \phi^2}\left(\frac{\partial \phi}{\partial \theta}\right)^2 + \frac{\partial u_0^i}{\partial \phi}\frac{\partial^2 \phi}{\partial \theta^2}\right] + \sin^{-1}\theta \frac{\partial^2 u_0^i}{\partial \omega^2}.
\]
Since $J(\alpha)$ has a minimum at $\alpha=0$, the one-sided derivative $J'(0^+)$ is non-negative (we cannot strengthen this into $J'(0)=0$ as $v_\alpha$ is not necessary differentiable on an open interval containing $\alpha=0$). We compute this derivative
\begin{align*}
 &(m-4)\frac{d}{d\alpha}J^i(\alpha)  =\\ 
 &\int_0^{\frac{\pi}{2}}\int_{\S^{m-2}} 2 K(\alpha,\omega,\phi)\frac{\partial}{\partial\alpha}K(\alpha,\omega,\phi)\sin^{m-4}\theta(\phi,\alpha)R^{m-4}(\phi,\alpha)\left|\frac{\partial \theta}{\partial \phi}\right|\\
 &\phantom{\int_0^{\frac{\pi}{2}}\int_{\S^{m-2}}} + |K(\alpha, \omega,\phi)|^2 \cdot\Biggl((m-4)\sin^{m-5}\theta\cos\theta\frac{\partial\theta}{\partial\alpha}R^{m-4}\left|\frac{\partial \theta}{\partial\phi}\right| \\
 &\phantom{\int_0^{\frac{\pi}{2}}\int_{\S^{m-2}}}\quad + (m-4)\sin^{m-4}\theta R^{m-5}\frac{\partial R}{\partial\alpha}\left|\frac{\partial \theta}{\partial\phi}\right|+ \sin^{m-4}\theta R^{m-4}\frac{\partial^2 \theta}{\partial\phi^2}\Biggr) \, d\omega \,d\phi,
\end{align*}
where
\[
 \begin{split}
 \frac{\partial}{\partial \alpha}K(\alpha, \omega,\phi) &= -\sin\theta(\phi,\alpha)\frac{\partial\theta}{\partial \alpha}\frac{\partial u_0^i}{\partial \phi}\frac{\partial \phi}{\partial \theta} + \cos\theta(\phi,\alpha)\frac{\partial u_0^i}{\partial \phi}\frac{\partial^2\phi}{\partial \theta\partial\alpha} \\
 &\quad+ \cos\theta(\phi,\alpha)\frac{\partial\theta}{\partial \alpha}\left[\frac{\partial^2 u_0^i}{\partial \phi^2}\left(\frac{\partial \phi}{\partial \theta}\right)^2 + \frac{\partial u_0^i}{\partial \phi}\frac{\partial^2 \phi}{\partial \theta^2}\right]\\
 &\quad +\sin\theta(\phi,\alpha)\left[2\frac{\partial^2 u_0^i}{\partial\phi^2}\frac{\partial\phi}{\partial\theta}\frac{\partial^2\phi}{\partial\theta\partial\alpha}+\frac{\partial u_0^i}{\partial \phi}\frac{\partial^3 \phi}{\partial \theta^2 \partial \alpha}\right]\\
 &\quad- \sin^{-2}\theta(\phi,\alpha)\cos\theta(\phi,\alpha)\frac{\partial\theta}{\partial\alpha}\frac{\partial^2 u_0^i}{\partial \omega^2}.
 \end{split}
\]

Using the following observations:
\begin{multicols}{2}
\begin{itemize}
\setlength\itemsep{0,5em}
    \item[(i)] $R(\phi,\alpha)\!\mid_{\alpha=0}=1$,
    \item[(ii)] $\left[\partial \theta / \partial \alpha\right]_{\alpha=0}=\sin\phi$,
    \item[(iii)] $\left[\partial \theta / \partial\phi \right]_{\alpha=0} = 1= \left[\partial\phi / \partial \theta\right]_{\alpha=0}$,
    \item[(iv)] $\left[ \partial R / \partial \alpha\right]_{\alpha=0} = -\cos\phi$,
    \item[(v)] $[\partial^2 \theta / \partial \phi\partial\alpha]_{\alpha = 0} = \cos\phi$,
    \item[(vi)] $[\partial^2\phi / \partial \theta \partial \alpha]_{\alpha=0}=-\cos\phi$,
    \item[(vii)] $[\partial^3 \phi / \partial \theta^2\partial\alpha]_{\alpha = 0}=\sin\phi$,
    \item[(viii)] $[\partial^2\phi / \partial\theta^2]_{\alpha=0} =0$,
    \item[(ix)] $\sin\theta(\phi,\alpha)\!\mid_{\alpha=0}=\sin \phi$,
    
\end{itemize}
\end{multicols}
\noindent
and letting $e(u_0) = \sum_{i=1}^k\left|\cos \phi \frac{\partial u_0^i}{\partial \phi}+\sin\phi\frac{\partial^2 u_0^i}{\partial \phi^2} + \sin^{-1}\phi \frac{\partial^2 u_0^i}{\partial \omega^2}\right|^2 \sin^{m-4}\phi$, we conclude that

\begin{align*}
 0\le (m-4)J'(0^+) &= -2\int_0^{\pi/2}\cos\phi\int_{\S^{m-2}}e(u_0) \, d\omega \,d\phi\\
                    &\quad + (m-4)\int_0^{\pi/2}\cos\phi\int_{\S^{m-2}}e(u_0) \, d\omega \,d\phi\\
                    &\quad - (m-4)\int_0^{\pi/2}\cos\phi\int_{\S^{m-2}}e(u_0) \, d\omega \,d\phi\\
                    &\quad + \int_0^{\pi/2}\cos\phi\int_{\S^{m-2}}e(u_0) \, d\omega \,d\phi\\
                    &= - \int_0^{\pi/2}\cos\phi\int_{\S^{m-2}}e(u_0) \, d\omega \,d\phi \le 0.
\end{align*}
Hence, $e(u_0) =0$ for almost all $(\varphi,\omega)$ and $u_0$ must be continuous, therefore constant.
\end{proof}

We will need the following lemma due to Scheven \cite[Lemma 4.2]{Scheven}.
\begin{lemma}\label{le:mintanmaps}
 Assume that $\widehat{v}\in W^{2,2}_{loc}(\R^m,\n)$ is a tangent map of a minimizing biharmonic map and for some $5\le k\le m$ it satisfies $sing(\widehat{v}) = \R^{m-k}\times \{0\}$ and $\partial_i \widehat{v}\equiv 0$ for all $1\le i\le m-k$. Then the restriction $v:=\widehat{v}\!\mid_{\{0\}\times\R^k}\in C^\infty(\R^k\setminus\{0\}, \n)$ is a minimizing biharmonic map and homogeneous of degree zero.
\end{lemma}

In the following proof of the boundary regularity by the above lemma we will get that the maps that appear in Federer dimension reduction argument are minimal. We will not repeat the whole argument, as it is known for experts. Instead we refer the interested reader to \cite[Theorem A.4.]{Simon-gmt} and in the case of harmonic maps \cite[pp. 332--334]{SU1}

\begin{proof}[Proof of Theorem \ref{thm:mainbih}]
We note that the boundary regularity of biharmonic maps follows for $m\le 3$ by Sobolev embedding and in the critical dimension $m=4$ is already known (see \cite{LammWang}). 
We follow the proof of \cite[Corollary 5.8., p. 579]{HLp}. 

For $m=5$ every map which is homogeneous of degree 0 map must be smooth away from the origin. 

For $m\ge 5$ we make an $(m-4)$ repeated formulation of boundary tangent maps (see \cite[Proof of Theorem II and IV, pp.333--334]{SU1}), until we obtain a boundary tangent map at a point $b\in T_1$ in the form $u_0(x,y) = v_0(y)$, where $(x,y)\in \R^{m-5}\times \R^5$ and $v_0$ is a map whose only discontinuity occurs at the origin. In this case, it follows from Lemma \ref{le:mintanmaps} that $v_0$ and hence $u_0$ is minimizing. By Lemma \ref{le:mintanmaps} $u_0$ is homogeneous of degree 0 and constant at $T_1$. Thus, by \ref{jednstzero} $u_0$ is constant.

In order to obtain $u_0$ we constructed a formulation of boundary tangent map, each time getting a sequence of maps converging strongly to a boundary tangent map. Now applying a diagonal sequence argument we extract a subsequence $\lambda_i$ and rescaled maps $u_{b, \lambda_i}$ which converge strongly to $u_0$ as $\lambda_i\searrow 0$.  
Therefore, because $u_0$ is constant, for each $\epsilon>0$ there exists a number $M>0$ such that for each $i>M$
 \begin{equation}
  \brac{\frac{\lambda_{i}}{2}}^{4-m}\int_{B^+_{\lambda_i}(b)}|\nabla^2 u|^2 \,dx < \epsilon, \qquad \brac{\frac{\lambda_{i}}{2}}^{2-m}\int_{B^+_{\lambda_i}(b)}|\nabla u|^2 \,dx < \epsilon.
 \end{equation}

We claim now that for every $\epsilon>0$ there exists $\widetilde{R}>0$ such that for each $\lambda<\widetilde{R}$ 
\begin{equation}\label{eq:epscond}
 \lambda^{4-m}\int_{B^+_\lambda(b)}|\nabla^2 u|^2 \,dx + \lambda^{2-m}\int_{B^+_\lambda(b)}|\nabla u|^2 \,dx < \epsilon.
\end{equation}

Indeed, assume on the contrary that there exists an $\epsilon>0$ such that for each $j\in\N$ there exists a $\lambda_j<\frac1j$ such that
\begin{equation}\label{contrary}
\begin{split}
 \epsilon &\le \lambda_{j}^{4-m}\int_{B^+_{\lambda_j}(b),}|\nabla^2 u|^2 \,dx + \lambda_{j}^{2-m}\int_{B^+_{\lambda_j}(b)}|\nabla u|^2 \,dx\\
 &\approx \int_{B^+_{\frac12}(b)} |\nabla^2 u_{\lambda_{j}}(y)|^2 \, dy + \int_{B^+_{\frac12}(b)} |\nabla u_{\lambda_{j}}(y)|^2 \, dy.
\end{split}
\end{equation}
But this contradicts the strong convergence of $u_{\lambda_n}$ in $W^{2,2}\big(B^+_{1/2},\n\big)$ to a constant map.

Now by Lemma \ref{le:glw41}, \eqref{eq:epscond} implies for an $r<\w{R}$ 
\begin{equation}\label{eq:morreyest2}
 \norm{\nabla^2 u}_{L^{2,m-4}(B_r^+(b))}^2 + \norm{\nabla u }_{L^{4,m-4}(B_r^+(b))}^4<C_1 \sqrt{\epsilon}. 
\end{equation}
Thus, by Theorem \ref{le:eps-reg} we finally conclude that $u\in C^\infty(\overline{B^+_{\frac r2}}(b),\n)$. 
\end{proof}

\appendix
 \section{}\label{a:biharmonic}
\begin{lemma}\label{a:secondderivativesbylaplacian}
There is a constant $C$ depending only on $m$ such that for any $0<r<R$ and any map $u\in W^{2,2}(B_R^+,\R^\ell)$ with vanishing $W^{2,2}$ trace on $T_R = \{x\in B_R\colon x_m=0\}$ we have
\begin{equation}\label{eq:appestimatesecondderivatives}
 \int_{B_r^+}|\nabla^2 u|^2 \dx \le 2 \int_{B_R^+} \brac{|\Delta u|^2 + \frac{C}{(R-r)^2}|\nabla u|^2} \dx.
\end{equation}
\end{lemma}

\begin{proof}
The proof is exactly as in \cite[Lemma A.1]{Scheven}. We choose the same cut-off function $\eta\in C_c^\infty(B_R,[0,1])$ such that $\eta\equiv 1$ on $B_r$ and $|\nabla \eta|<\frac{C}{(R-r)}$. If we assume that $u\in C^\infty(B_R^+,\R^l)$ and integrate twice by parts the following integral
\[
 \int_{B_R^+} \eta^4|\Delta|^2 \dx,
\]
the boundary term will vanish on the flat part of $\partial B_R^+$, because the $W^{2,2}$ trace of $u$ vanishes there. It will also vanish on the curved part of the boundary, for $\eta$ vanishes there. Thus,
\begin{equation*}
 \begin{split}
  \int_{B_R^+}\eta^4 u_{x_i x_i}u_{x_j x_j}\dx &= - 4\int_{B_R^+}\eta^3\eta_{x_j} u_{x_i x_i}u_{x_j}\dx - \int_{B_R^+}\eta^4 u_{x_i x_i x_j}u_{x_j}\dx\\
  &= - 4\int_{B_R^+}\eta^3\eta_{x_j} u_{x_i x_i}u_{x_j}\dx + 4\int_{B_R^+}\eta^3\eta_{x_i} u_{x_i x_j}u_{x_j}\dx\\
  &\quad  + \int_{B_R^+}\eta^4 u_{x_i x_j}u_{x_i x_j}\dx.
 \end{split}
\end{equation*}

Hence,
\begin{equation*}
\begin{split}
 \int_{B_R^+}\eta^4 |\nabla^2 u|^2 \dx & \le \int_{B_R^+}\eta^4 |\Delta u|^2 \dx + C\int_{B_R^+}\eta^3|\nabla\eta||\nabla u||\nabla^2 u|\dx \\
&\le \int_{B_R^+}\eta^4 |\Delta u|^2 \dx + \frac12 \int_{B_R^+}\eta^4 |\nabla^2 u|^2\dx \\
&\quad + \frac{C}{(R-r)^2}\int_{B_R^+}\eta^2|\nabla u|^2 \dx.
\end{split}
\end{equation*}
The desired inequality for smooth $u$ follows by subtracting $\frac12 \int_{B_R^+}\eta^4 |\nabla^2 u|^2\dx$ from both sides. Now an approximation argument yields the the same argument for $W^{2,2}$ maps. 
\end{proof}

The next Lemma shows that by boundary monotonicity formula a bound in $W^{2,2}$ implies a bound in the Morrey space $L^{2,m-4}$. The proof is almost identical to the proof in the interior case, but as the boundary monotonicity formula yields an additional term we sketch the proof below. 
\begin{lemma}\label{le:appboundednessinmorrey}
 Let $u\in W^{2,2}(B^+_4,\n)$ be a minimizing biharmonic map with boundary value $\vp$ as in \eqref{eq:dirichletboundary} and let$\norm{u-\vp}_{W^{2,2}(B^+)}<\infty$. Let $\w{u}$ be the reflection of $u-\vp$ given in \ref{def:reflection}, then 
\begin{equation}\label{eq:appmorreybound}
 \sup_{y\in B,\ \rho <1} \rho^{4-m}\int_{B_\rho(y)}|\nabla^2 \w{u}|^2 \dx \le C \int_{B_2}|\nabla^2 \w{u}|^2 \dx + \w{C},
\end{equation}
for constants $C=C(m)$ and $\w{C} = \w{C}(m,\n)$.
\end{lemma}

\begin{proof}[Proof of Lemma \ref{le:appboundednessinmorrey}]
 We give the necessary modification of \cite[Lemma A.2]{Scheven}.
 
 \noindent
 We note that since $u$ satisfies the boundary monotonicity formula \eqref{eq:boundarymono} we have for $\w{u}$ and $a\in T_1$ the following 
\begin{equation}\label{eq:reflboundarymono}
\begin{split}
\rho^{4-m}&\int_{B_\rho(a)}|\nabla^2 \wu|^2 \dx + Ce^{C\rho} R^+_u(a,\rho)\\
&\le C\brac{e^{Cr}r^{4-m}\int_{B_r(a)}|\nabla^2 \wu|^2 \dx + e^{Cr} R^+_u(a,r) + Cre^{Cr}}.
\end{split}
\end{equation}

Let $0<s<1/8$ be given. By Fubini theorem we may choose good radii $\rho<r$ with $s\le\rho\le2s<\frac12\le r\le1$ such that
\[
 \begin{split}
 \rho^{3-m}\int_{B_\rho(a)}|\nabla \w{u}|^2 \dh &\le C s^{2-m}\int_{B_{2s(a)}}|\nabla \w{u}|^2 \, dx\\
  \rho^{5-m}\int_{B_\rho(a)}|\nabla^2 \w{u}|^2 \dh &\le C s^{4-m}\int_{B_{2s(a)}}|\nabla^2 \w{u}|^2 \, dx\\
  \int_{B_r(a)}\brac{|\nabla^2 \wu|^2 + |\nabla \wu|^2}\dh &\le C\int_{B_1(a)}\brac{|\nabla^2 \wu|^2 + |\nabla\wu|^2} \, dx,
 \end{split}
\]
where the constant $C$ depends only on the dimension $m$.

One can easily observe that
\[
\begin{split}
 \left|R^+_{u,\vp}(a,\tau)\right| &\le C\tau^{4-m}\int_{\partial B^+_\tau(a)}\brac{|\nabla^2 u||\nabla u| + \frac{1}{\tau}|\nabla u|^2}\dh\\
 &\le C\tau^{4-m}\int_{\partial B_\tau(a)}\brac{|\nabla^2 \wu||\nabla \wu| + \frac{1}{\tau}|\nabla \wu|^2}\dh.
\end{split}
 \]

Combining this observation with \eqref{eq:reflboundarymono} we get
\[
\begin{split}
\rho^{4-m}\int_{B_\rho(a)}|\nabla^2 \w{u}|^2 \dx &\le Ce^{Cr}r^{4-m}\int_{B_r(a)}|\nabla^2 \w{u}|^2 \dx + Cre^{Cr}\\
&\quad+ C\rho^{4-m}\int_{\partial B_\rho(a)}\brac{|\nabla^2 \w{u}||\nabla \w{u}| + \frac{1}{\rho}|\nabla \w{u}|^2}\dh\\
&\quad + C e^{Cr}r^{4-m}\int_{\partial B_r(a)}\brac{|\nabla^2 \w{u}||\nabla \w{u}| + \frac{1}{r}|\nabla \w{u}|^2}\dh\\
\end{split}
\]
Thus, since $s<\rho<2s$ and by Young's inequality with $\epsilon$
\begin{equation}\label{eq:cdszacowanlemata2}
 \begin{split}
  s^{4-m}\int_{B_s(a)}|\nabla^2 \w{u}|^2 \dx &\le C\rho^{4-m}\int_{B_\rho(a)}|\nabla^2 \w{u}|^2 \dx\\
  & \le \frac14 (2s)^{4-m} \int_{B_{2s}(a)}|\nabla^2 \w{u}|^2\ dx + C s^{2-m}\int_{B_{2s}(a)}|\nabla \wu|^2 \dx \\
  & \quad + C\int_{B_1(a)}|\nabla^2 \wu|^2 \dx + C\int_{B_1(a)}|\nabla\wu|^2 \dx + C.
 \end{split}
\end{equation}
Next, we proceed exactly as in \cite{Scheven}. Observe that by Nirenberg's interpolation inequality
\[
  \|\nabla f\|_{L^4(\Omega)}^2 \le C(\Omega) \|f\|_{L^\infty(\Omega)}\|f\|_{W^{2,2}(\Omega)}
 \]
we have after a few transformations
\begin{equation}\label{eq:nirenbest}
 C\tau^{2-m}\int_{B_\tau(y)}|\nabla \wu|^2 \dx \le \frac14 \tau^{4-m}\int_{B_{\tau}(y)}|\nabla \wu|^2\dx + \w{C},
\end{equation}
where $\w{C}$ is a constant dependent on the target manifold $\n$. Applying \eqref{eq:nirenbest} into \eqref{eq:cdszacowanlemata2} for $\tau=1$ and $\tau = 2s$, denoting $\widehat{H}(\tau):= \tau^{4-m}\int_{B_\tau(y)}|\nabla^2 \wu|^2 \dx$, we arrive at 
\[
\widehat{H}(s)\le \frac12 \widehat{H}(2s) + C\widehat{H}(1) + \w{C}.
\]
for all $0<s<\frac14$. 

Thus, for all small $\sigma>0$
\[
 \sup_{\sigma<s<1} \widehat{H}(s) \le \sup_{\sigma <s <1/4} \widehat{H}(s) + C\widehat{H}(1) \le \frac{1}{2} \sup_{\sigma<s<1/4} \widehat{H}(2s)+ C\widehat{H}(1) + \w{C}.
\]
Since $\sigma>0$ the term $\frac{1}{2} \sup_{\sigma<s<1/4} \widehat{H}({2}s)$ is finite and can be absorbed by the left hand side of the inequality giving
\[
 \sup_{\sigma <s<1}\widehat{H}(\rho)\le C\widehat{H}(1) + \w{C}.
\]
The estimate is independent of $\sigma>0$ and thus the claimed inequality follows. 
\end{proof}

\begin{remark}
 In the last proof we did not need a higher order reflection. An odd reflection is enough to ensure that if $u-\vp\in W^{2,2}_0(B^+,\R^\ell)$, then the reflected map is in $W^{2,2}(B,\R^\ell)$.
\end{remark}

\subsection*{Acknowledgment}
This work is based on the PhD thesis of the author under supervision of Pawe\l{} Strzelecki. The author would like to thank Pawe\l{} Strzelecki for many helpful discussions.

The work was partially supported by the National Science Center in Poland via grant NCN 2015/07/B/ST1/02360.

\bibliographystyle{plain}%
\bibliography{bibliografia}%

\end{document}